\documentclass[a4paper,12pt]{article}

  \usepackage[utf8]{inputenc}
  \usepackage[T1]{fontenc}
  \usepackage[utf8]{inputenc}
  \usepackage{lmodern}
  \usepackage[top=2cm,bottom=2cm, left=2.3cm, right=2.3cm]{geometry}
  \usepackage{amsmath, amssymb}
  \usepackage{hyperref}
  \usepackage{color}
  \usepackage{graphicx}
  \usepackage{tikz}
  \usepackage{tikzrput}
  \usepackage{verbatim}
  \usepackage{mathrsfs}
  \usepackage{cancel}
  \usepackage[affil-it]{authblk}

  \usepackage{amsthm}
  \usepackage{enumerate}

  \theoremstyle{plain}
  
  \newcounter{thI}
  \newtheorem{thmInt}[thI]{Theorem}
    \newtheorem*{theorem-non}{Theorem}

  \newtheorem{lem}{Lemma}[section]
  \newtheorem{prop}{Proposition}[section]

  \newcounter{cntasmp}

  \newtheorem{assumption}[cntasmp]{Assumption}

  \theoremstyle{definition}
  \newtheorem{defi}{Definition}

  \theoremstyle{remark}
  \newtheorem{rmq}{Remark}[section]
  \newtheorem{rmqInt}{Remark}[]
  
  \def\ptn(#1)(#2)(#3){\fill[color=#3](#1)circle(#2)}

\begin{document}

  \renewcommand{\proofname}{Proof}
  \renewcommand\thethI{\arabic{thI}} 
   \author{Laurent Dietrich%
  \thanks{Electronic address: \texttt{laurent.dietrich@math.univ-toulouse.fr}}}
\affil{Institut de Mathématiques de Toulouse ; UMR5219 \\ Université de Toulouse ; CNRS \\ UPS IMT, F-31062 Toulouse Cedex 9, France}
    \title{Existence of travelling waves for a reaction-diffusion system with a line of fast diffusion}
  
  \maketitle

\begin{abstract}
We prove existence and uniqueness of travelling waves for a reaction-diffusion system coupling a classical reaction-diffusion equation in a strip with a diffusion equation on a line. To do this we use a sequence of continuations which leads to further insight into the system. In particular, the transition occurs through a singular perturbation which seems new in this context, connecting the system with a Wentzell type boundary value problem.
\end{abstract}

%\tableofcontents
\newpage

\section{Introduction}
This paper deals with the following system with unknowns $c > 0, \phi(x), \psi(x,y)$ : 
\begin{alignat*}{1}
&\begin{cases}
-d\Delta \psi+ c\partial_x \psi = f(\psi)\text{ for }(x,y)\in \Omega_L := \mathbb R\times]-L,0[ \\
d\partial_y \psi(x,0) = \mu \phi(x) - \psi(x,0)\text{ for }x\in \mathbb R \\
-d\partial_y \psi(x,-L) = 0 \text{ for }x\in \mathbb R\\
\end{cases}
\\ 
& -D\phi''(x) + c\phi'(x) = \psi(x,0) - \mu \phi(x)\text{ for }x\in \mathbb R
\end{alignat*}
along with the uniform in $y$ limiting conditions
$$\mu \phi, \psi \to 0 \text{ as }x\to -\infty$$
$$\mu \phi, \psi \to 1 \text{ as } x\to +\infty$$
Those equations will be represented from now on as the following diagram 
\begin{equation}
\label{normal}
  \begin{tikzpicture}
  \draw (-6,0) -- (6,0) node[pos=0.5,below] {\small{$d\partial_y \psi = \mu \phi - \psi$}} node[pos=0.5,above] {$-D\phi'' + c\phi' = \psi - \mu \phi$};

  \node at (4.3,0.33) {$\phi\to 1/\mu$};
  \node at (-4.3,0.33) {$0 \leftarrow \phi$};

  \node at (0,-1.5) {$- d\Delta \psi + c\partial_x \psi = f(\psi)$};

  \draw (-6,-3) -- (6,-3) node[pos=0.5,above] {\small{$\partial_y \psi= 0$}};

 % \draw[<->] (-4.3,-3) -- (-4.3,0) node[pos=0.5,left] {$L$};

  \node at (4.3,-1.5) {$\psi \to 1$};
  \node at (-4.3,-1.5) {$0 \leftarrow \psi$};

  \end{tikzpicture}
\end{equation}
If $(c,\phi,\psi)$ is a solution of \eqref{normal}, then $(\phi(x-ct),\psi(x-ct,y))$ is a travelling wave solution connecting the states $(0,0)$ and $(1/\mu,1)$ for the following reaction-diffusion system 
\begin{equation}
\label{readi}
  \begin{tikzpicture}
  \draw (-6,0) -- (6,0) node[pos=0.5,below] {\small{$d\partial_y v = \mu u - v$}} node[pos=0.5,above] {$\partial_t u - Du''  = v - \mu u$};

  \node at (0,-1.5) {$\partial_t v - d\Delta v  = f(v)$};

  \draw (-6,-3) -- (6,-3) node[pos=0.5,above] {\small{$\partial_y v= 0$}};

 % \draw[<->] (-4.3,-3) -- (-4.3,0) node[pos=0.5,left] {$L$};

  \end{tikzpicture}
\end{equation}

This system in the whole half-plane $y<0$ was introduced by Berestycki, Roquejoffre and Rossi in \cite{BRR} to model the influence of a line of fast-diffusion on propagation. Under the assumption $f(v)=v(1-v)$ (which will be referred to from now on as a KPP type non-linearity) the authors showed the following : when $D \leq 2d$, propagation of the initial datum along the $x$ direction occurs at the classical KPP velocity $c_{KPP} = 2\sqrt{df'(0)}$, but when $D > 2d$ it occurs at some velocity $c^*(\mu,d,D) > c_{KPP}$ which satisfies $$\lim_{D\to+\infty} \frac{c^*}{\sqrt{D}} = c > 0$$
%\begin{theorem-non}(\cite{BRR})
%\begin{enumerate}[i)]
%\item Spreading. There is an asymptotic speed of spreading $c_* = c_*(\mu,d,D) > 0$ such that the following is true. Let the initial datum $(u_0,v_0)$ be $\geq 0$ and $\not\equiv (0,0)$. Then :
%\begin{itemize}
%\item for all $c> c_*$, $\lim_{t\to+\infty} \sup_{|x|\geq ct} (u(x,t), v(x,y,t)) = (0,0)$.
%\item for all $c < c_*$, $\lim_{t\to+\infty} \inf_{|x|\geq ct} (u(x,t), v(x,y,t)) = (1/\mu,1)$.
%\end{itemize}
%\item  The spreading velocity. If $d$ and $\mu$ are fixed the following holds true.
%\begin{itemize}
%\item If $D \leq 2d$, then $c_*(\mu,d,D) = c_{KPP} = 2\sqrt{df'(0)}$
%\item If $D > 2d$ then $c_*(\mu,d,D) > c_{KPP}$ and $\lim_{D\to+\infty} c_*(\mu,d,D)/\sqrt{D}$ exists and is a positive real number.
%\end{itemize}
%\end{enumerate}
%\end{theorem-non}

In the present work, we take another viewpoint to extend these results to more general non-linearities. Indeed, the KPP assumption $$f(v) \leq f'(0)v$$ enables a reduction of the question to algebraic computations. For more general reaction terms (e.g. bistable, or ignition type), propagation is usually governed by the travelling waves. As a consequence, it is necessary to investigate the existence, uniqueness, and stability of solutions of \eqref{normal} in order to generalise this result, which will be seen through the velocity $c(D)$ of the solution.

We make the following assumption 
\begin{assumption}
\label{asf}
 $f:[0,1] \to \mathbb R $ is a smooth non-negative function, $f = 0$ on $[0,\theta]\cup\{1\}$ with $\theta > 0$, $f(0) = f(1) = 0$, and $f'(1) < 0$. For convenience we will still call $f$ an extension of $f$ on $\mathbb R$ by zero at the left of $0$ and by its tangent at $1$ (so it is negative) at the right of $1$. 
\end{assumption}
\begin{figure}[h!]
 \centering \def\svgwidth{200pt} 
 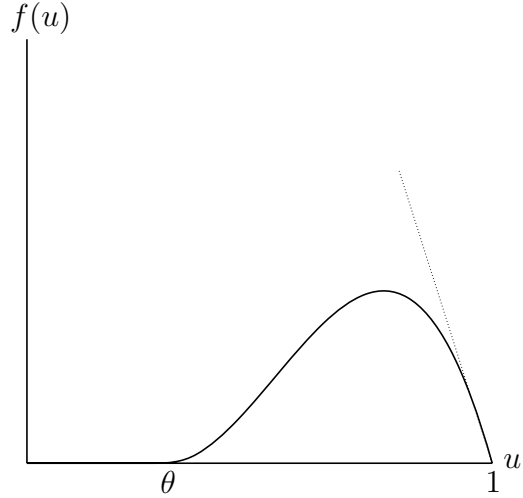 \caption{Example $f = \mathbf 1_{u>\theta}(u-\theta)^2(1-u)$} 
\end{figure}
Our objective is to study \eqref{normal} by a continuation method : we will show that \eqref{normal} can be reduced, through "physical" steps, to the classical one dimensional equation 
$$ - \psi'' + c\psi' = f(\psi) $$
$$ \psi(-\infty) = 0, \psi(+\infty) = 1$$
This is the simplest model in the description of propagation of premixed flames (see the works of Kanel \cite{KA}).
More precisely, the steps we will follow are : 
\begin{enumerate}[1)]
\item First, a good way of reaching a unique equation is to have $\mu\phi=\psi$ on the boundary $y=0$. To achieve that, we divide the exchange term by a small $\varepsilon > 0$ and send $\varepsilon \to 0$.  Setting $\mu\phi = \psi(x,0) + \varepsilon\phi_1$ we get after a simple computation the limiting model for $\psi$ : $$-\frac{D}{\mu}\partial_{xx}\psi + \frac{c}{\mu}\partial_x\psi = -\phi_1$$ and so this limit is a singular perturbation that sends $(S)_\varepsilon$ :
\begin{equation}
\tag{$S_\varepsilon$}
\label{Seps}
\begin{tikzpicture}
\draw (-6,0) -- (6,0) node[pos=0.5,below] {\small{$d\partial_y\psi = (\mu\phi - \psi(x,0))/\varepsilon$}} node[pos=0.5,above] {$-D\phi'' + c\phi' = (\psi(x,0) - \mu\phi)/\varepsilon$} ;

\node at (4.3,0.5) {$\mu\phi \to 1$};
\node at (-4.3,0.5) {$0 \leftarrow \phi$};

\node at (0,-1.5) {$-d\Delta\psi + c\partial_x\psi = f(\psi)$};

\draw (-6,-3) -- (6,-3) node[pos=0.5,above] {-\small{$\partial_y\psi = 0$}};

\node at (4.3,-1.5) {$\psi \to 1$};
\node at (-4.3,-1.5) {$0 \leftarrow \psi$};

\end{tikzpicture}
\end{equation}
to a unique equation with a Wentzell boundary condition that we call $(W)_s$ (with $s=1$):
\begin{equation}
\label{frontswentzell}
\tag{$W_s$}
\begin{tikzpicture}
\draw (-6,0) -- (6,0) node[pos=0.5,below] {\small{$d\partial_y\psi = s\left(\frac{D}{\mu}\partial_{xx}\psi- \frac{c}{\mu}\partial_x\psi\right)$}};
\node at (0,-1.5) {$-d\Delta\psi + c\partial_x\psi = f(\psi)$};

\draw (-6,-3) -- (6,-3) node[pos=0.5,above] {-\small{$\partial_y\psi = 0$}};

\node at (4.3,-1.5) {$\psi \to 1$};
\node at (-4.3,-1.5) {$0 \leftarrow \psi$};

\end{tikzpicture}
\end{equation}

\item By sending the parameter $s\in[0,1]$ to $0$, we can pass from $(W)_s$ to $(W)_0$, which is a Neumann problem, which is known to have a unique velocity $c_0 > 0$ and a unique smooth profile $\psi_0(x)$ (up to translations) as solutions. Existence is due to Kanel \cite{KA} and uniqueness to Berestycki-Nirenberg \cite{BN90}.
\end{enumerate}

We will show the following :

\begin{thmInt}(Existence for the Wentzell model)

\label{w}
 There exists $c_w >0$ and $\psi_w \in \mathcal C^{3,\alpha}(\Omega_L)$ for some $0<\alpha <1$, solution of $(W)_1$ obtained by continuation from $(c_0,\psi_0)$ that satisfies $0 < \psi_w < 1$, and $\psi_w$ is increasing in the $x$ direction. Moreover, if $(\underline c,\overline \psi)$ is a classical solution of $(W)_1$, we have $\underline c = c$ and there exists $r\in\mathbb R$ such that $\overline\psi(\cdot + r) = \psi_w(\cdot)$. 
\end{thmInt}

\begin{thmInt}(Transition from Wentzell to the system)

\label{eps} 
 There exists $\varepsilon_0 > 0$ small such that for $0 < \varepsilon < \varepsilon_0$, $(S)_{\varepsilon}$ has a solution $(c_{\varepsilon},\phi_{\varepsilon},\psi_{\varepsilon})\in\mathbb R_+^*\times \mathcal C^{2,\alpha}(\mathbb R)\times \mathcal C^{2,\alpha}(\Omega_L)$ obtained by continuation from $\left(c_w,\frac{1}{\mu}\psi_w(\cdot,0),\psi_w\right)$ that satisfies $0 < \phi_{\varepsilon} < \frac{1}{\mu}$, $0<\psi_\varepsilon<1$, and $\phi_\varepsilon$ and $\psi_\varepsilon$ are increasing.
\end{thmInt}

\begin{rmqInt}
\label{rmksing}
We wish to emphasise on this result : it consists in a singular perturbation between a system of two unknowns and a scalar boundary value problem of the Wentzell type. This is a non-standard relaxation that appears to be new in this context. Also, observe that we had to pay the loss of one derivative for solving this problem.
\end{rmqInt}

\begin{thmInt}(Existence for the full system)

\label{sys}
 There exists $(c,\phi,\psi)\in\mathbb R_+^*\times \mathcal C^{2,\alpha}(\mathbb R)\times \mathcal C^{2,\alpha}(\Omega_L)$ a solution of $(S)_1$ obtained by continuation from $(c_{\varepsilon_0},\phi_{\varepsilon_0},\psi_{\varepsilon_0})$ that satisfies $0 < \phi < \frac{1}{\mu}$, $0<\psi <1$, and $\phi$ and $\psi$ are increasing. Moreover, if $(\underline c,\overline \phi, \overline\psi)$ is a classical solution of $(S)_1$, we have $\underline c = c$ and there exists $r\in\mathbb R$ such that $\overline\phi(\cdot + r) = \phi(\cdot)$ and  $\overline\psi(\cdot + r) = \psi(\cdot)$.
\end{thmInt}

The organisation of the paper is as follows :
\begin{itemize}
 \item In Section \ref{firstprop} we show some common a priori properties for solutions of \eqref{Seps} or \eqref{frontswentzell}. In particular we deal with the uniqueness questions.
 \item In Section \ref{neumanntowentzell} we prove Theorem \ref{w}.
 \item Section \ref{epsto1} proves Theorem \ref{sys}, provided Theorem \ref{eps}, by slight modifications of Section \ref{neumanntowentzell}.
 \item Finally, we postponed the proof of Theorem \ref{eps} in Section \ref{pertsing} because of its particularity (see Remark \ref{rmksing}).
 \end{itemize}

\section{First properties}
\label{firstprop}

\subsection{A priori bounds, monotonicity, uniqueness}

This section is devoted to the proofs of a priori properties. As noticed in \cite{BRR}, the system \eqref{readi} has the structure of a monotone system, which provides a maximum principle. The elliptic counterpart of this maximum principle holds for \eqref{Seps} or \eqref{frontswentzell} and it will be our main tool along with the sliding method (\cite{BN91}) throughout the current section.

\begin{lem}
\label{phipsi}
Let $(c,\phi,\psi)$ solve some \eqref{Seps}. Then $$\text{inf }\psi \leq \mu\phi \leq \text{sup }\psi$$
\end{lem}
\begin{proof}
Because of its uniform limits as $x\to\pm\infty$, necessarily $\psi$ is bounded and $|\psi|_\infty \geq 1$. Because of the limits of $\phi$, either $\phi \leq \frac{1}{\mu}$ or $\phi-\frac{1}{\mu}$ reaches a positive maximum. But then at this maximum $\phi' = 0$ and $\phi'' \leq 0$ which means $\psi - \mu \phi \geq 0$, so in every case $\phi \leq \frac{ |\psi|_\infty}{\mu}$. The other inequality is similar.
\end{proof}

\begin{prop}
\label{bounds}
Let $(c,\phi,\psi)$ be a solution of some \eqref{Seps}, then $$0 < \mu\phi, \psi < 1$$
Similarly, if $(c,\psi)$ solves some \eqref{frontswentzell}, $0 < \psi < 1$.
\end{prop}

\begin{proof}
Suppose there exists a point $(x_0,y_0)$ where $\psi(x_0,y_0)>1$. Then, because $\psi$ is assumed to have limits as $x\to\pm\infty$, we see that $\psi-1$ must reach a positive maximum somewhere.  But $\psi-1$ satisfies locally at this point $-d\Delta(\psi-1) + c\partial_x(\psi-1) < 0$. This point cannot be in $\Omega_L$ by the strong maximum principle, since $\psi-1$ would be locally constant which is impossible by looking at the equation. So it has to be on the boundary. It cannot be on $y=-L$ because of the Hopf lemma. So it has to be on $y=0$ and by the Hopf lemma, $\mu\phi > \psi$ at this point, what is impossible because of Lemma \ref{phipsi}. So $\psi \leq 1$, and then $\phi \leq \frac{1}{\mu}$. 

Now knowing these bounds and that the solutions are not constants, comparison with $0$ by the strong maximum principle gives $\phi < \frac{1}{\mu}, \psi < 1$.

Finally, since $\psi < 1$, $f(\psi) \geq 0$ and the strong maximum principle along with Lemma \ref{phipsi} gives $\psi > 0$ and then $\phi > 0$.

The same proof holds for equation \eqref{frontswentzell}, the case $y_0 = 0$ being treated by the sole Hopf lemma thanks to the sign of $\partial_{xx}\psi$ and the nullity of $\partial_x \psi$ on an extremum of $\psi$. 

\end{proof}

We turn now to the monotonicity of the fronts, using the sliding method of \cite{BN90} and simplified in \cite{V} in the travelling waves context. We start with a fundamental lemma which asserts that we can slide a supersolution above a subsolution by translating it enough to the left. This lemma is valid for any reaction term such that $f(0) = f(1) = 0$ and $f'(0), f'(1) \leq 0$.

\begin{defi}
We call a super (resp. sub) solution of some \eqref{Seps} or \eqref{frontswentzell} a function which satisfies those equations with the $=$ signs replaced by $\geq$ (resp. $\leq$), and the uniform limits replaced by some constants $\geq 0, \geq 1$ (resp. $\leq 0, \leq 1$).
\end{defi}

\begin{lem}
\label{fonda}
Let $(c,\underline\phi,\underline\psi)$ be a subsolution of some \eqref{Seps} and $(c,\overline\phi,\overline\psi)$ a supersolution. Then there exists $r_0$ such that for all $r \geq r_0$, $\overline\phi^r := \overline\phi(r + \cdot),  \overline\psi^r := \overline\psi(r + \cdot,\cdot)$ satisfy
$$\overline\phi^r > \underline\phi, \overline\psi^r > \underline\psi$$
The same holds with $\underline\psi$, $\overline\psi$ resp. sub and supersolution of some \eqref{frontswentzell}.
\end{lem}

\begin{proof}
We present only the proof for \eqref{Seps}, the case \eqref{frontswentzell} being simpler. We also assume that $\underline\psi, \mu\underline\phi, \overline\psi, \mu\overline\phi \to_{x\to -\infty} 0$ and $\underline\psi, \underline\phi, \overline\psi, \mu\overline\phi \to_{x\to +\infty} 1$, the case of different limits being considerably simpler. 

First we show that by translating enough, we have the desired order on some $x \geq a$ (which is trivial if the sub and the supersolution have different limits to the right). Fix $\varepsilon > 0$ small enough such that $f' \leq 0$ on $[0,\varepsilon] \cup [1-\varepsilon,1]$. Because of the conditions at $\pm\infty$ it is clear that there exists $r_1 > 0$ and $a>0$ large enough such that 
\begin{alignat}{1}
&\mu\underline\phi, \underline\psi > 1-\varepsilon\text{ on }x\geq a\label{fonda_cond1} \\ 
&\text{For all }r\geq r_1\text{, }\mu\overline\phi^r(a) > \mu\underline\phi(a)\text{, } \overline\psi^r(a,y) > \underline\psi(a,y\label{fonda_cond2}) \\ 
&\text{For all }r\geq r_1\text{, }\mu\overline\phi^r > 1-\varepsilon\text{, }\overline\psi^r > 1-\varepsilon\text{ on }x\geq a \label{fonda_cond3}
\end{alignat}
Conditions \eqref{fonda_cond2} and \eqref{fonda_cond3} are obtained simultaneously by taking $r_1$ large enough. We assert that this suffices to have 
$$\mu\overline\phi^r > \mu\underline\phi, \overline\psi^r > \underline\psi \text{ on } x\geq a\text{ for all }r \geq r_1$$
Indeed, call $U = \overline\phi^r - \underline\phi$ and $V = \overline\psi^r - \underline\psi$. Then in $[a,+\infty[\times[-L,0]$ we have the following~:
\begin{equation*}
  \begin{tikzpicture}
  \draw (-5.5,0) -- (7.5,0) node[pos=0.5,below] {\small{$d\partial_y V + V \geq \mu U$}} node[pos=0.5,above] {$-U'' + cU' +  \mu U \geq V$};

  \node at (6.8,0.33) {$U \to 0$};

  \node at (1.5,-2) {$LV := -d\Delta V  + c\partial_x V \geq f(\overline\psi^r) - f(\underline\psi) $};

  \draw (-5.5,-4) -- (7.5,-4) node[pos=0.5,above] {\small{$-\partial_y V \geq 0$}};

 \draw[blue] (-5.5,-4) -- (-5.5,0) node[pos=0.5,left,text=blue]{$V > 0$};

  \node at (6.8,-2) {$V \to 0$};
  
  \ptn(-5.5,0)(0.1)(blue);
  \node[text=blue] at (-5.5,0.4) {$U > 0$};

  \filldraw [fill=gray!20,draw=black]
(-1.5,0) arc (180:0:-1)
(-3.5,0) -- (-1.5,0)
(-2.5,-0.4) node[below] {$B_-$}
(-2.5,0) node[above] {$p$};

\filldraw [fill=gray!20,draw=black]
(-2.5,-4) arc (180:0:1)
(-2.5,-4) -- (-0.5,-4)
(-1.5,-3.9) node[above] {$B_+$}
(-1.5,-4) node[below] {$p$};

\filldraw [fill=gray!20,draw=black]
(-3,-2) arc (360:0:1)
(-4,-1.8) node[above] {$B$}
(-4,-2) node[right] {$p$};

\ptn(-2.5,0)(0.05)(black);
\ptn(-1.5,-4)(0.05)(black);
\ptn(-4,-2)(0.05)(black);

  \end{tikzpicture}
\end{equation*}
Suppose there is a point where $V < 0$. Then because of its limits, $V$ reaches a minimum $V(p) = m < 0$ at some $p \in ]a,+\infty[\times[-L,0]$. 
\begin{itemize}
\item Case 1 : $p\in ]a,+\infty[\times]-L,0[$. By continuity, there is a ball $B$ around $p$ such that on $B$, $1 - \varepsilon < \overline\psi^r < \underline \psi$. On $B$, $LV \geq f(\overline\psi^r) - f(\underline\psi) \geq 0$ since $f(s)$ is decreasing on $s \geq 1-\varepsilon$. By the strong maximum principle, $V \equiv m$ in $B$. Thus, $\{V = m\}$ is open. Being trivially closed and being non-void, it is all of $[a,+\infty[\times[-L,0]$ which is a contradiction.

\item Case 2 : $p$ lies on $y=-L$. Again, by continuity there is a half ball $B_+$ just as $B$ in the previous case. By the Hopf lemma, since $\partial_y V(p) \geq 0$ necessarily $V = m$ is also reached in the interior of $B_+$, and we fall in case 1.

\item Case 3 : $p$ lies on $y=0$. Taking another half ball $B_-$ as above, either we fall in case 1 or $\mu U < m < 0$ at $x_p$. But this is impossible also. Indeed, $U$ would reach a negative minimum somewhere, but looking at the equation it satisfies at that minimum, $\mu U \geq V \geq m$.
\end{itemize}
Every case leading to a contradiction, we conclude that $\overline\psi^r - \underline\psi \geq 0$ on $x\geq a$. The strong maximum principle applied on $U$ yields now $\overline\phi^r - \underline\phi \geq 0$ on $x\geq a$ and that the orders are strict.

We do the same thing for $x \leq b$ up to the following subtlety : we can only ask for the following conditions 
\begin{alignat}{1}
&\mu\underline\phi, \underline\psi < \varepsilon\text{ on }x\leq b\label{fonda_cond21} \\ 
&\text{For all }r\geq r_2\text{, }\mu\overline\phi^r(b) > \mu\underline\phi(b)\text{, } \overline\psi^r(b,y) > \underline\psi(b,y\label{fonda_cond22}) 
\end{alignat}
Of course an equivalent of condition \eqref{fonda_cond3} is not available here : we cannot ask to put the supersolution everywhere below $\varepsilon$ on $x\leq b$ whereas before it was automatic to put it above $1-\varepsilon$. Nonetheless, the exact same proof as above works, since on an eventual minimum of $\overline\psi^r - \underline\psi$ we would have this order for free : $\overline\psi^r < \underline\psi < \varepsilon$. 

Finally, taking $r_3 = \max(r_1,r_2)$ we end up with the supersolution above the subsolution on all $x\not\in]b,a[$ : thanks to the uniform limits of $\overline\phi^{r_3}, \overline\psi^{r_3}$ to the right, we just have to translate enough again to cover the compact region left.

For the case of equation \eqref{frontswentzell}, just observe that case 3 is similar to case 2, since on a minimum, $\partial_{xx} V \geq 0, \partial_x V = 0$.

\end{proof}

\begin{rmq}
The use of the maximum principle in the proof above could be simplified, since on $x \geq a$ we know the sign of $$k(x,y) = -\frac{f(\overline \psi^r) - f(\underline\psi)}{\overline\psi^r - \underline\psi} \in L^\infty$$
we could apply it directly in all $\Omega_L$ with the operator $-d\Delta + c\partial_x + k$. Nonetheless, this is not true any more on $x \leq b$ and this is the reason why we chose the proof above.
\end{rmq}

\begin{prop}
\label{mono}
Let $(c,\phi,\psi)$ be a solution of some \eqref{Seps}, then $$\phi', \partial_x \psi > 0$$
Similarly, if $(c,\psi)$ solves some \eqref{frontswentzell}, $\partial_x\psi > 0$.
\end{prop}
\begin{proof}
Use Lemma \ref{fonda} with the solution serving as the sub and the supersolution at the same time : we can translate some $\mu\phi^r, \psi^r$ over $\mu\phi, \psi$. Call $r$ the $\inf$ of such $r_0$, i.e. slide back until the solutions touch (which clearly happens since at $r_0 = 0$ they are the same). Monotonicity will be proved if we show that 
$$r = 0$$
Suppose by contradiction that $r > 0$. By continuity
\begin{alignat*}{1}
U := \phi^r - \phi \geq 0 \\
V:= \psi^r - \psi \geq 0
\end{alignat*}
Moreover, $U, V$ solves 
\begin{equation*}
  \begin{tikzpicture}
  \draw (-7.5,0) -- (7.5,0) node[pos=0.5,below] {\small{$d\partial_y V + V = \mu U$}} node[pos=0.5,above] {$-U'' + cU' +  \mu U = V$};

  \node at (6.8,0.33) {$U \to 0$};
  \node at (-6.8,0.33) {$0 \leftarrow U$};

  \node at (0,-1.5) {$-d\Delta V  + c\partial_x V - k(x,y)V = 0$};

  \draw (-7.5,-3) -- (7.5,-3) node[pos=0.5,above] {\small{$-\partial_y V = 0$}};
  
  \node at (6.8,-1.5) {$V \to 0$};
  
  \node at (-6.8,-1.5) {$0 \leftarrow V$};
  \end{tikzpicture}
\end{equation*}
with $k(x,y) = -\frac{f(\psi^r) - f(\psi)}{\psi^r - \psi} \in L^\infty$. The strong maximum principle and Hopf's lemma for comparison with a minimum that is $0$ gives that $U,V > 0$ (otherwise $\phi, \psi$ would be periodic in $x$, which is impossible). Then by continuity, for any compact $$K_a = [-a,a]\times\left( [-a,a]\times[-L,0]\right) $$ we can still translate a bit more to the right while keeping the order : $$\mu\phi^{r-\varepsilon_a} > \mu\phi\text{, }\psi^{r-\varepsilon_a} > \psi\text{ on }K_a$$ for some small $\varepsilon_a > 0$. Now just do this with $a$ large enough so that on $x\leq a$, $\psi < \varepsilon$ and on $x\geq a$, $\psi > 1-\varepsilon$ (and so $\psi^{r-\varepsilon_a}$ too, even on $x \geq a - \varepsilon_a$). Then the exact same proof as in Lemma \ref{fonda} applies to conclude that $$\mu\phi^{r-\varepsilon_a} > \mu\phi\text{, }\psi^{r-\varepsilon_a} > \psi$$ everywhere, which is a contradiction with the minimality of $r$.

We now know that $\phi$ and $\psi$ are increasing in $x$, that is $\phi',\partial_x \psi \geq 0$. To conclude that $\phi', \partial_x\psi > 0$ everywhere, just differentiate \eqref{Seps} or \eqref{frontswentzell} with respect to $x$ and apply the strong maximum principle and Hopf's lemma for comparison with a minimum $0$. We emphasise on the fact that this result is valid up to the boundary of $\Omega_L$.

\end{proof}

The proof above gives directly the following rigidity result and his corollaries :
\begin{prop}(Uniqueness among sub or supersolutions.)

\label{unisol}
Fix $c > 0$. If \eqref{Seps} has a solution, then every supersolution or subsolution is a translated of this solution. The same holds for \eqref{frontswentzell}.
\end{prop}
\begin{proof}
Denote $(\phi,\psi)$ the solution mentioned and $(\overline\phi, \overline\psi)$ an eventual supersolution. Let $r, U, V$ be as in the proof of Proposition \ref{mono} (this time $r$ exists thanks to the limit conditions~: at some point the supersolution and the solution touch). We end up with either $U, V > 0$ or $U, V \equiv 0$. The first case is impossible for the exact same argument as in Proposition \ref{mono} and this concludes the proof.
\end{proof}

\begin{prop}(Uniqueness of the velocity and the profiles up to translation.)
\label{uniqueness} \ 
\begin{enumerate}
\item There is a unique $c_\varepsilon \in\mathbb R$ such that \eqref{Seps} can have solutions. The same holds for $c_s$ with \eqref{frontswentzell}.
\item Solutions of \eqref{Seps} are unique up to $x$-translations, i.e. if $(c,\phi_1,\psi_1)$ and $(c,\phi_2,\psi_2)$ are solutions of \eqref{Seps}, then there exists $r \in \mathbb R$ such that $$\phi_2(\cdot+r) = \phi_1(\cdot), \psi_2(\cdot+r,\cdot) = \psi_1(\cdot,\cdot)$$ The same holds for \eqref{frontswentzell}.
\end{enumerate}
\end{prop}
\begin{proof}\
\begin{enumerate}
\item Call $(c,\phi,\psi)$ and $(\overline c,\underline{\phi}, \underline{\psi})$ two solutions such that $\overline c > c$. Observe that thanks to monotonicity :
\begin{alignat*}{1}
&-d\Delta\underline\psi + c\partial_x\underline\psi = f(\underline\psi) + (c-\overline c)\partial_x\underline\psi < f(\underline\psi) \\
&-\underline\phi'' + c\underline\phi' + \mu\underline\phi = \underline\psi (c-\overline c)\underline\phi' < \underline\psi
\end{alignat*}
so that $(\underline{\phi}, \underline{\psi})$ is a subsolution of equation \eqref{Seps} with $c$ and is not a solution, which is impossible thanks to Proposition \ref{unisol}. The case of equations \eqref{frontswentzell} is treated in a similar way, just observe that 
$$d\partial_y\underline\psi = s\left(\frac{D}{\mu}\partial_{xx}\underline\psi- \frac{\overline c}{\mu}\partial_x\underline\psi\right) < s\left(\frac{D}{\mu}\partial_{xx}\underline\psi- \frac{c}{\mu}\partial_x\underline\psi\right)$$

\item Apply Proposition \ref{unisol}, knowing that a solution is also a subsolution.
\end{enumerate}
\end{proof}

\subsection{Uniform bounds for $c$}
\label{c_bounds}

Our continuation method will need compactness on $c > 0$ if we want to extract a solution from a sequence of solutions. Getting an upper bound will depend on finding supersolutions of \eqref{Seps} or \eqref{frontswentzell}. Then a lower bound will follow easily via an argument of \cite{BLL90}.

\begin{prop}
\label{sectioncmax}
There exists $c_{max} > 0$ such that any solution $(c_\varepsilon,\phi_\varepsilon,\psi_\varepsilon)$ of \eqref{Seps} satisfies $$c_\varepsilon < c_{max}$$ The same holds for \eqref{frontswentzell}.
\end{prop}
\begin{proof}
Observe that if 
\begin{equation*}
\begin{cases}
-Dr^2 + cr \geq 0 \\
-dr^2 + cr \geq \text{Lip} f
\end{cases}
\end{equation*}
Then $(e^{rx},\mu e^{rx})$ is a supersolution of \eqref{Seps} which is not a solution. The first inequation gives $r = \alpha c/D$ with $\alpha \in[0,1]$ and the best choice for minimising $c$ in the second one is $\alpha = D/(2d)$ or $\alpha = D$ depending on $D \lessgtr 2d$. More precisely,

\begin{equation}
\label{cmax}
c_{max} =
\begin{cases}
\displaystyle{2\sqrt{d\text{Lip} f}} \text{ if }D\leq 2d \\
\displaystyle{\sqrt{\frac{D^2}{D-d}\text{Lip} f}} \text{ if }D\geq 2d 
\end{cases}
\end{equation}
The exact same computation holds for \eqref{frontswentzell}.
\end{proof}

\begin{rmq}
Note that at $s=0$, $D$ does not appear in the equation so the lowest $c_{max}$ is valid but exhibits a discontinuity as soon as $s > 0$ (if $D > 2d$). Of course, since this is only an upper bound it is not a problem. Actually, this is just technical : if we had done the continuation from Neumann to oblique and from oblique to Wentzell in two steps, this discontinuity would not be since the comparison would occur between $sD \lessgtr 2d$.
\end{rmq}

\begin{prop} 
\label{cmin}
There exists $c_{min} > 0$ such that any solution of \eqref{Seps} satisfies $$c_\varepsilon \geq c_{min} > 0$$ The same holds for equation \eqref{frontswentzell}.
\end{prop}
\begin{proof}
In this proof we forget about the $\varepsilon$ for the sake of notations. We integrate the equation for $\psi$ on $\Omega_{L,M}:=[-M,M]\times[-L,0]$ using integration by parts. For the first term, we have 

\begin{equation*}
\begin{split}
\int_{\Omega_{L,M}} -d\Delta\psi &= \int_{\partial\Omega_{L,M}} -d\partial_\nu\psi \\
& = \int_{[-L,0]} d\partial_x\psi(-M,y)dy - \int_{[-L,0]} d\partial_x\psi(M,y)dy + \int_{-M}^M -d\partial_y\psi(x,0)dx \\
& = \int_{[-L,0]} d\partial_x\psi(-M,y)dy - \int_{[-L,0]} d\partial_x\psi(M,y)dy + \int_{-M}^M (\psi(x,0)-\mu\phi(x))dx
\end{split}
\end{equation*}
Using elliptic estimates and dominated convergence, we see that the first two terms go to zero as $M\to\infty$, which gives
$$ \int_{\Omega_{L}} -d\Delta\psi = \int_{\mathbb R} (\psi(x,0)-\mu\phi(x)) dx = \int_{\mathbb R} (-D\phi''+ c\phi') dx = \frac{c}{\mu}$$
thanks to elliptic estimates on $\psi$.

For the second term, we have 
$$\int_{\Omega_{L,M}} c\partial_x\psi = \int_{[-L,0]} c\psi(M,y)dy - \int_{[-L,0]}c\psi(-M,y)dy \to cL$$ by dominated convergence. We thus have 
\begin{equation}
\label{formulec}
c=\frac{1}{L+1/\mu} \int_{\Omega_L} f(\psi)
\end{equation}
Now, any solution satisfies $-d\Delta\psi + c\partial_x\psi = f(\psi)$ in $\Omega_L$ with $c$ and $f(\psi)$ bounded independently of $c$ by the constant $M_0 = \max(d, c_{max}, \sup f)$. Thus on the ball $B$ of centre $(0,-L/2)$ and radius $L/4$, standard $L^2$ elliptic estimates and the Sobolev embedding give for any $0 < \beta < 1$

$$ |\psi|_{C^{\beta}(B)} \leq C_1\left( |\psi|_{L^2(2B)} + |f(\psi)|_{L^2(2B)} \right) \leq C_1|2B|(1 + \sup f^2) \leq C_2$$
with $C_2$ independent of $\varepsilon$ and where $2B$ denotes $B$ with doubled radius, and $|2B|$ its measure. We just proved that all solutions share a modulus of continuity independent of $\varepsilon$ on the ball $B$. Since $f$ is Lipschitz, the same holds for $f(\psi)$.

Now normalise the solutions by translation so that $$\psi\left( 0,\frac{-L}{2}\right)  = \frac{1+\theta}{2}$$ The previous estimate enables us to choose a radius $r_0 > 0$ small enough that depends only on $C_2$ and $\text{Lip} f$ such that $f(\psi) \geq \frac{1}{2}f( \frac{1+\theta}{2})$ on the ball $r_0B$. This implies the lower bound $$\int_{\Omega_L} f(\psi) \geq |r_0B|\dfrac{1}{2}f\left( \dfrac{1+\theta}{2}\right)  > 0$$ that gives the existence of
$$c_{min} = \frac{|r_0B|}{2(L+1/\mu)} f\left(\dfrac{1+\theta}{2}\right)$$
that depends only on $d, \mu, L, c_{max}, \sup f, \text{Lip} f$.

For equation \eqref{frontswentzell}, the exact same proof holds since \eqref{formulec} is replaced by $$c=\frac{1}{L+s/\mu}\int_{\Omega_L} f(\psi) \geq \frac{1}{L+1/\mu}\int_{\Omega_L} f(\psi) $$
for $s\in[0,1]$.
\end{proof}

\section{From Neumann to Wentzell}
\label{neumanntowentzell}
Set $$P_{W} = \{s \in [0,1] \quad | \quad (W)_s\text{ has a solution}\}$$ The main goal of this section is to prove that $P_{W}$ is open and closed in $[0,1]$, as in \cite{BM09}. We will proceed as follows :
\begin{itemize}
\item We already know that $0\in P_{W}$ so that $P_W \neq \varnothing$.
\item In Section \ref{pwclosed} we prove that $P_{W}$ is closed, using the bounds on $c$ from Section \ref{firstprop} and a regularity result up to the boundary for \eqref{Seps} or \eqref{frontswentzell}. 

We emphasise on a small but interesting technical difficulty : in the context of \eqref{frontswentzell}, no standard $L^p$ estimates up to the boundary appear to be in the literature. As a consequence, we had to use a weak Harnack inequality up to the Wentzell boundary to prove the Hölder regularity of $f(\psi)$, which is needed to use the Schauder estimates of \cite{LT91}.

\item In Section \ref{pwopen} we prove that $P_W$ is open, by perturbing \eqref{frontswentzell} for $s$ close to some $s_0 \in P_W$ in a weighted space where we can apply the implicit function theorem.
\end{itemize}
Together with the uniqueness properties of Section \ref{firstprop}, this will prove Theorem \ref{w}.

\subsection{$P_{W}$ is closed}
\label{pwclosed}
In this section we consider a sequence $(s^n) \subset P_{W}$ that converges to $s^\infty \in ]0,1]$ and we want to show that $(W_{s^\infty})$ has a solution thanks to the compactness results we already obtained. Denote $(c_n,\psi_n)$ a solution of $(W_{s^n})$. Throughout all this section we break the translation invariance by making the normalisation
\begin{equation}
\label{normalisation}
 \max_{x\leq 0} \psi_n = \theta
\end{equation}
We also drop a finite number of terms of the sequence $(s_n)$ so that for all $n \geq 0$, $s_n > \frac{s_\infty}{2} > 0$, which will be needed to ensure the uniform ellipticity of the boundary operator in \eqref{frontswentzell} so that we can use the elliptic estimates up to the Wentzell boundary.

By Section \ref{c_bounds} we can extract from $c^n$ some subsequence still denoted $c^n$ that satisfies

\begin{equation}
\label{cnconverge}
\lim_{n\to +\infty} c^n = c^\infty \in [c_{min},c_{max}]
\end{equation}

We now derive global Schauder estimates for \eqref{frontswentzell} from the standard local ones of \cite{LT91}. We describe the argument exhaustively for once because we will refer to it later for the more complicated case of \eqref{Seps}. We chose deliberately to use that $|\psi_s| \leq 1$ only at the end to give the inequality in its full generality, since the proof will serve later purposes.

\begin{prop}
\label{schauderwentzell}
There exists $\alpha > 0$ and a constant $C_{Sch}=C(D,d,c_{max},\text{Lip} f,L,\mu)$ such that for all $n \geq 0$
\begin{equation}
\label{estiw}
|\psi_n|_{\mathcal C^{2,\alpha}(\Omega_L)} \leq C_{Sch}\left(|\psi_n|_{L^\infty(\Omega_L)} \right) \leq C_{Sch}
\end{equation}

\end{prop}
\begin{proof}
We only prove a local estimate near $y=0$, the rest of the strip being treated similarly but with classical interior Schauder estimates or up to the Neumann boundary (see \cite{GT}, Theorem 6.29). 

Schauder estimates up to the Wentzell boundary are already proved in \cite{LT91}, but of course they need a bound on the $\mathcal C^\alpha$ norm of the data $f(\psi)$ (or on the bounded coefficient $-f(\psi)/\psi$ after rewriting the equation). Usually, this not a problem and, for example, can be derived from $W^{2,p}$ estimates up to the boundary.

Nonetheless, no such $L^p$ estimates appear to be in the literature concerning Wentzell boundary conditions. We overcome this technical difficulty by using directly a $\mathcal C^\alpha$ (for some small $\alpha > 0$) estimate up to the boundary (see \cite{Luo91},~Theorem~2) which relies on a weak Harnack inequality up to the boundary (see \cite{Luo93}) : in other words, the Krylov-Safonov inequality of \cite{KS} is valid up to a Wentzell boundary.

Call $B_- \subset \Omega_L$ (resp. $2B_-$) some half-ball of centre $(x,0)$ and radius $\varepsilon > 0$ (resp. $2\varepsilon$) small. By the references above there exists $\alpha > 0$ and some $C_\alpha$ depending only on $c_{max}, d, D, \mu, \varepsilon$ such that
$$|\psi_n|_{\mathcal C^\alpha(B_-)} \leq C_\alpha \text{Lip} f |\psi_n|_{L^\infty(2B_-)}   $$ 
since $f$ is Lipschitz and $f(0)=0$. This yields $$|f(\psi_n)|_{\mathcal C^\alpha(B_-)} \leq C_\alpha \left( \text{Lip} f\right) ^2 |\psi_n|_{L^\infty(2B_-)}$$ and then by plugging this in the Schauder estimates up to the Wentzell boundary (\cite{LT91}, Theorem 1.5) :
$$|\psi_n|_{\mathcal C^{2,\alpha}(B_-)} \leq C_W\left( |\psi_n|_{L^\infty(2B_-)}\right)$$
for some $C_W = C(d,D,c_{max},\mu, \text{Lip} f)$. 

To obtain the global estimate, just use the global $L^\infty$ bound and observe that the above estimate does not depend on the position of $B_-$.

\end{proof}

\begin{rmq}
Of course we can now iterate the Schauder estimate for any $\mathcal C^{k,\alpha}$ provided enough regularity on $f$. Namely, if $f$ has $k$ Lipschitz derivatives, $\psi_n$ is uniformly in $\mathcal C^{k+2,\alpha}$ for every $0 < \alpha < 1$. 
\end{rmq}
%\begin{alignat*}{1}
%|\psi^n|_{\mathcal C^{2,\alpha}(\Omega_L)} \leq C\left(|\psi^n|_{L^\infty(\Omega_L)} + |\mu\phi^n|_{\mathcal C^{1,\alpha}(\Omega_L)}\right)
%\end{alignat*}

Using \eqref{estiw} with Ascoli's theorem and the process of diagonal extraction for every $[-N,N]\times [-L,0]$, we get a subsequence still denoted $\psi_n$ that converges in $\mathcal C^2_{loc}(\overline{\Omega_L})$ to a function $\psi_\infty \in \mathcal C^2(\overline{\Omega_L})$. Remembering \eqref{cnconverge} we can pass to the limit in $(W_{s^n})$ to get that $(c_\infty, \psi_\infty)$ solves $(W_{s^\infty})$ apart from the limiting conditions. This is the aim of the following lemmas.

\begin{prop}
\label{expdecay}
 $\psi^\infty(x,\cdot)$ converges uniformly to $0$ as $x\to -\infty$.
\end{prop}

\begin{proof}
This relies on a comparison with the exponential supersolution already computed in Proposition \ref{sectioncmax}. Observe that thanks to \eqref{normalisation}, any solution of \eqref{frontswentzell}+\eqref{normalisation} satisfies $f(\psi_s) \equiv 0$ on $x\leq 0$. As a consequence $$\overline p_s:= \theta e^{r_sx}$$ where $$r_s = \frac{c_s}{\max(d,D)} \geq \frac{c_{min}}{\max(d,D)} =: r$$ is a supersolution of \eqref{frontswentzell} on $x\leq 0$. 

Since $\overline p_s - \psi_s$ is non-negative on $x=0$, goes uniformly to $0$ as $x\to -\infty$, satisfies a Neumann boundary condition on $y=-L$, a Wentzell boundary condition on $y=0$ and $-d\Delta  u + c\partial_x u \geq 0$ inside $x < 0$, the strong maximum principle and Hopf's lemma give for all $x\leq 0$ :
\begin{equation}
 \psi_s (x,y) \leq \theta e^{r_s x} \leq \theta e^{rx} \label{compgauche}
\end{equation}
The result is obtained by taking $s=s^n$ and making $n \to +\infty$ in the above inequality.  
\end{proof}

The right limit condition is obtained by simple computations already done in \cite{BLL90} in the Neumann case, we adapt them here.
\begin{prop}
\label{rightlimit}
 $\psi^\infty(x,\cdot)$ converges uniformly to $1$ as $x\to +\infty$.
\end{prop}
Since bounds and monotonicity pass to the $\mathcal C^2$ limit, we have $0 \leq \psi^\infty \leq 1$, as well as $\psi^\infty_x \geq 0$. As a consequence there exists $\beta(y) \leq 1$ such that $\psi^\infty(x,y) \to \beta(y)$ as $x\to +\infty$. Let us define the functions $\psi^\infty_j(x,y) = \psi^\infty(x+j,y)$ in $[0,1]\times[-L,0]$ for every integer $j$. Elliptic estimates  and Ascoli's theorem tell us that up to extraction, $\psi^\infty_j \to \delta$ in the $\mathcal C^1$ sense for a $\mathcal C^1$ function $\delta$. By uniqueness of the simple limit, $\beta = \delta \in \mathcal C^1$. So $\psi^\infty_j$ lies in a compact set of $\mathcal C^1( [0,1] \times [-L,0])$ and has a unique limit point $\beta\in \mathcal C^1$ : then it converges to it in the $\mathcal C^1$ topology.

\begin{lem}
\label{integrals}
$\int_{\Omega_L} f(\psi^\infty) < +\infty$ and $\int_{\Omega_L} |\nabla \psi^\infty|^2 < +\infty$
\end{lem}
\begin{proof}
For the first integral we integrate $(W_{s^\infty})$ on $Q_M:=[0,M]\times[-L,0]$ using integration by parts. We obtain 

\begin{equation*}
\begin{split}
\int_{Q_M} f(\psi^\infty) =& \int_{-L}^0 -d\partial_x\psi^\infty(M,y)dy + \int_{-L}^0 d\partial_x\psi^\infty(0,y)dy  \\ 
& + \displaystyle\frac{s^\infty}{\mu}\left(c_\infty(\psi^\infty(M,0)- \psi^\infty(0,0)) - D(\partial_x \psi^\infty(M,0)- \partial_x\psi^\infty(0,0))\right)  \\
& + \int_{-L}^0 c^\infty\psi^\infty(M,y)dy - \int_{-L}^0 c^\infty\psi^\infty(0,y)dy
\end{split}
\end{equation*}
which can be written as $$ \int_{Q_M} f(\psi^\infty) = A(M) - A(0)$$ with $$A(m) = c^\infty\int_{-L}^0 \psi^\infty(m,y)dy + \frac{s^\infty}{\mu}\left(c_\infty \psi^\infty(m,0) - D\partial_x\psi^\infty(m,0)\right) - d\int_{-L}^0 \partial_x\psi^\infty(m,y)dy$$ Since the first two terms in $A(m)$ are bounded (thanks to $\psi_\infty \leq 1$ and elliptic estimates), $\int_{Q_M} f(\psi^\infty) \to +\infty$ would mean that $d\int_{-L}^0 \partial_x\psi^\infty(m,y)dy \to -\infty$ as $M\to+\infty$ which is impossible since it is the derivative of the bounded function $m\mapsto d\int_{-L}^0 \psi^\infty(m,y)dy$.

For the second integral, we proceed in the same manner, but integrating the equation multiplied by $\psi^\infty$ and integrating by parts, we get 
\begin{equation}
\label{gradpsiinf2}
d\int_{Q_M} |\nabla\psi^\infty|^2 = B(M) - B(0) + \int_{Q_M} f(\psi^\infty)\psi^\infty - \frac{s_\infty D}{\mu}\int_0^M (\partial_x \psi^\infty)^2(x,0) dx
\end{equation}
with \begin{alignat*}{1}
B(m) = & -\frac{c^\infty}{2}\int_{-L}^0 {\psi^\infty}(m,y)^2dy - \frac{c^\infty s^\infty}{2\mu}{\psi^\infty}(m,0)^2 + \frac{s_\infty D}{\mu} (\partial_x\psi^\infty \psi^\infty)(m,0) \\ 
& -d\int_{-L}^0 (\psi^\infty \partial_x\psi^\infty)(m,y)dy 
\end{alignat*}
The third term in \eqref{gradpsiinf2} is bounded thanks to $0 \leq \psi^\infty \leq 1$ and what we just saw. The last one is non-positive. The first two terms in $B(m)$ are bounded, the third one also by elliptic estimates, so $\int_{Q_M} |\nabla\psi^\infty|^2 \to +\infty$ would mean $d\int_{-L}^0 (\psi^\infty \partial_x\psi^\infty)(m,y)dy \to -\infty$ as $M\to\infty$, which is impossible since it is the derivative of the bounded function $m\mapsto -d\int_{-L}^0 \frac{1}{2}{\psi^\infty}(m,y)^2dy$. 
 \end{proof}

\begin{proof}[End of the proof of Lemma \ref{rightlimit}]
We now turn back to the study of the right limit. The second integral in Lemma \ref{integrals} being finite, necessarily $\nabla\beta = 0$ \footnote{or else we have uniformly $|\nabla u_j|^2 \to g$, with $g > \delta > 0$ in $[y_0-\varepsilon,y_0+\varepsilon]$ and the second integral would be greater than $\int_{[R,\infty)\times[y_0-\varepsilon,y_0+\varepsilon]} \delta/2 = +\infty$ for $R$ large enough.}. So $\beta$ is a constant. Moreover, $$0 \leq \theta \leq \text{max}_{[-L,0]} \psi^\infty(0,y) \leq \beta \leq 1$$ We also have $f(\beta)=0$ because of the finiteness of the first integral\footnote{$\int_{Q_\infty}f(\psi^\infty) = \sum_{j=0}^\infty \int_{Q_1} f(\psi^\infty_{2j})$ so $\int_{Q_1} f(\psi^\infty_{2j}) \to 0$, but this $\to \int_{-L}^0 f(\beta(y))dy$ too.} so $\beta = \theta$ or $\beta = 1$. Suppose by contradiction that $\beta=\theta$. Then $f(\psi^\infty) \equiv 0$, and integrating the equation satisfied by $\psi^\infty$ on $[-m,m]\times[0,L]$ just as above and making $m\to+\infty$ yields
$$0 = A(\infty) - A(-\infty) = c^\infty\left(L+\frac{s^\infty}{\mu}\right)\theta$$
since $\partial_x\psi^\infty(m,y)\to 0$ uniformly in $y$ as $x\to\pm\infty$. This is of course, impossible, since $c^\infty>c_{min}>0$ and $\theta > 0$. 
\end{proof}

As a conclusion, $\psi^\infty$ satisfies all the desired properties, and we have proved that $P_{W}$ is closed.

\subsection{$P_{W}$ is open}
\label{pwopen}

This part is about applying the implicit function theorem to some function
$F(s,c,\psi)$ in order to get a solution for $s > s^0$ close to a value $s^0$ of the
parameter for which we already have a solution $c^0,\psi^0$. In this section we
take $\mu=1$ without loss of generality to clarify the diagrams. 
We set 
$$\psi=\psi^0 +(s-s^0)\psi^1\text{, }c=c^0+(s-s^0)c^1$$
where $s\in [s^0, s^0+\delta]$, $\delta > 0$ small to be fixed later.
After a simple but tedious computation,  we get that the corresponding equation for $\psi^1,c^1$ is :

\begin{center}
\begin{tikzpicture}
\draw (-7,0) -- (7,0) node[pos=0.5,below] {$\mathcal W\psi^1 = -(c^0+c^1s)\partial_x\psi^0 + D\partial_{xx}\psi^0 - (s-s^0)(c^0+c^1s)\partial_x\psi^1 + (s-s^0)D\partial_{xx}\psi^1$};

\node at (0,-1.5) {$\mathcal L\psi^1 + c^1\partial_x\psi^0 =
R(s-s^0,c^1,\psi^1)$};

\draw (-7,-3) -- (7,-3) node[pos=0.5,above] {$\partial_y\psi^1 = 0$};

\end{tikzpicture}
\end{center}
where $$\mathcal W =  d\partial_y + c^0s^0\partial_x - s^0D\partial_{xx}$$
and $$\mathcal L = -d\Delta  + c^0\partial_x - f'(\psi^0)$$ and $R$
being a function that goes to $0$ as $s\to s^0$ and decays quadratically in the variables $\psi^1, c^1$ in a setting that will be defined later\footnote{$R(s-s^0,c^1,\psi^1)
= -(s-s^0)c^1\partial_x\psi^1 + (s-s^0)\frac{f''(\psi^0)}{2}(\psi^1)^2 +
(s-s^0)^2\frac{f'''(\psi^0)}{6}(\psi^1)^3 + \cdots
= (s-s^0)\mathcal O(c^1, \psi^1 )$, the $\mathcal O$ being in $\mathbb R \times \mathcal C^{1,\alpha}$ norm.}. 

We will solve the order $1$ problem, i.e. the one obtained by taking $s=s^0$	
and then we will apply the implicit function theorem in a good functional
setting to obtain the existence of a solution to the above problem for $s$ close
to $s^0$. The upper boundary condition should be seen as close to a fixed non-homogeneous Wentzell boundary condition. That is why we first need some
information about the operator $\mathcal L$ with Wentzell condition.
\begin{center}
\begin{tikzpicture}
\draw (-5,0) -- (5,0) node[pos=0.5,below] {$\mathcal Wg=0$};

\node at (0,-1.5) {$\mathcal Lg = 0$};

\draw (-5,-3) -- (5,-3) node[pos=0.5,above] {$-\partial_y g = 0$};

\end{tikzpicture}
\end{center} It is well known (see \cite{SAT76},\cite{JMR92}) that this operator is not Fredholm in the usual spaces of bounded uniformly continuous functions due to the degeneracy of $f$ in the range $[0,\theta]$. The way to circumvent this difficulty is to endow the space with a weight that sees the exponential decay of the solutions as $x\to -\infty$.

%In order for this operator to have the suitable properties to solve the
%problem, we construct a suitable Banach space as in Sattinger \cite{SAT76} and
%Roquejoffre \cite{JMR92} by endowing
%$\mathcal{C}^{2,\alpha}(\Omega_L)$ with a weighted norm
%that sees the exponential decay.  

\begin{defi}
\label{defw}
Let \begin{equation}
\label{r}
r = \frac{c_{min}}{4\max(d,D)}
\end{equation} so that $-dr^2+ c^0r \geq 0$ and $-Dr^2+c^0r \geq 0$ (the $4$ will serve later purposes, see Lemma \ref{weightedlinf}). Define $w$ to be a smooth function such that 
\begin{equation}
w(x) = 
\begin{cases}
e^{rx} \text{ for }x < 0\\
2 \text{ for } x > 1\\
\end{cases}
\end{equation}
and that is concave and increasing on $0 < x <1$. 
Define also $w_1 = 1/w$.
\end{defi}

\begin{defi} Let $$\mathcal C^\alpha_w(\Omega_L) = \{u \in \mathcal{C}^{\alpha}(\Omega_L)
 \mid w_1u \in \mathcal{C}^{\alpha}(\Omega_L)\}$$ and $$X = \mathcal{C}^{2,\alpha}_w(\Omega_L)$$ the set of $\mathcal C^2$ functions on $\Omega_L$ whose derivatives up to order two are in $\mathcal C^\alpha_w(\Omega_L)$. We endow $X$ with the norm $$|u|_X
= |w_1u|_{\mathcal{C}^{2,\alpha}}$$
\end{defi}

$X$ is clearly a Banach space, which contains $\psi^0$. Indeed, at the left of $\Omega_L$, $w_1\psi^0$ satisfies a linear homogeneous Wentzell problem. Thus, the $\mathcal C^{2,\alpha}_w$ estimate directly comes from the Schauder estimates of \cite{LT91} by the $L^\infty$ estimate for $w_1\psi^0$, which was already proved to get the left-limit condition, in Proposition \ref{expdecay}. At the right of $\Omega_L$, $w_1$ is a bounded smooth function and being $\mathcal C^{2,\alpha}_w$ here is equivalent to being $\mathcal C^{2,\alpha}$. Now we have :

%We should also have in mind the fact that the same phenomenon appears when $x\to\infty$ : solutions tend actually exponentially fast to $1$, so by elliptic estimates their derivatives to zero, and so we can integrate them on $\Omega_L$ or on $y=0$ (this will be used in the later computations and the proof of this fact is contained in the proof \ref{rightlimitcd}).

\begin{lem}
\label{lemJMR}
$\mathcal L$ has closed range and there exists $X_1 \simeq R(\mathcal L)$ a closed subspace of $X$ and $Y_2 \simeq N(\mathcal L)$ such that 
$$X=N(\mathcal L)\oplus X_1$$
$$Y=R(\mathcal L)\oplus Y_2$$
Moreover $N(\mathcal L) = N(\mathcal L^2) = \mathbb R \partial_x{\psi_0}$. Finally, denote $\mathcal L^*$ the adjoint of $\mathcal L$. Then $N(\mathcal L^*)$ is one dimensional too. Calling $e^*$ the unique generator that satisfies $$<e^*, \partial_x\psi^0> = 1$$ we get that $e^*$ is a positive measure that happens to be a smooth positive function, solving 
$$\mathcal L^*e^* = \left(-d\Delta -
c^0\partial_x - f'(\psi^0)\right)e^* = 0$$ endowed with the dual boundary conditions 
\begin{alignat*}{1}
d\partial_ye^* - c^0s^0\partial_xe^* - s^0D\partial_{xx}e^* &= 0\text{ on }y=0\\ 
\partial_ye^*&=0\text{ on }y=-L 
\end{alignat*}
Moreover $e^*$ is bounded on $x > 0$ and has at most $Ce^{-rx}$ growth as $x\to -\infty$.
\end{lem}

\begin{proof}
The proof will be postponed to the last paragraph of this section. It all relies on the fact that $\mathcal L$ is a Fredholm operator of index $0$ on the weighted space $X$.
\end{proof}

Now we want to transform the problem into a fixed Wentzell homogeneous problem. We do this by creating an auxiliary function $\tilde\psi(s,c^1,v)$ such that we search for $\psi^1$ as $$\psi^1 = \tilde\psi(s,c^1,v) + v$$ where $\tilde\psi(s,c^1,v)$ solves for $A>0$ large enough

\begin{center}
\begin{tikzpicture}
\draw (-7.5,0) -- (7.5,0) node[pos=0.5,above] {$\mathcal Wu = D\partial_{xx}\psi^0-(c^0+c^1s)\partial_x\psi^0 - (c^0+c^1s)(s-s^0)\partial_x (u+v) + D(s-s^0)\partial_{xx}(u+v)$};

\node at (0,-1) {$\mathcal Lu + Au = 0$};

\draw (-7.5,-2) -- (7.5,-2) node[pos=0.5,below] {$\partial_y u = 0$};

\end{tikzpicture}
\end{center}

\begin{lem}
\label{lemmec2alphaw}
 Such a function exists and satisfies $\tilde\psi\in \mathcal C^1([0,1]\times \mathbb R \times \mathcal{C}^{2,\alpha}_w(\overline{\Omega_L}); \mathcal{C}^{2,\alpha}_w(\overline{\Omega_L}))$.
\end{lem}

\begin{proof}
For $A>|f'(\psi^0)|_\infty$ it is known that the above problem has a unique solution that lies in $\mathcal{C}^{2,\alpha}(\overline{\Omega_L})$ provided $v \in \mathcal{C}^{2,\alpha}(\overline{\Omega_L})$, since this gives that the data for the Wentzell condition lies in $\mathcal C^{\alpha}(\mathbb R)$ (see theorem 1.6 in \cite{LT91} along with the remark at its end). What is important to show is that if $v$ lies in the weighted space, the solution $u$ is in it too. On $x\geq 0$, $w_1u$ is trivially $\mathcal C^{2,\alpha}$ as the product of a smooth bounded function and a $\mathcal C^{2,\alpha}$ function. The only problem might come from unboundedness at $x \to -\infty$. 
 In other words, we need to show that $u$ decays like $Ce^{rx}$ as $x\to -\infty$. We see that $w_1u$ satisfies an elliptic problem too, so conversely by Schauder estimates the problem is reduced to showing this $L^\infty$ bound for $w_1u$ on $x<0$. More precisely : we see that $w_1u$ solves :

\begin{center}
\begin{tikzpicture}
\draw (-6,0) -- (6,0) node[pos=0.6,above] {$(d\partial_y - sD\partial_{xx} + a_1\partial_x + a_2)(w_1u) = w_1\varphi \in \mathcal C^{1,\alpha}(\mathbb R)$};

\node at (0,-1) {$(-d\Delta + b_1 \partial_x + b_2)(w_1u) = 0$};

\draw (-6,-2) -- (6,-2) node[pos=0.5,below] {$\partial_y (w_1u) = 0$};

\filldraw [fill=gray!20,draw=black]
(-4.5,-0.2) node[below] {$B$}
(-3.5,0) arc (180:0:-1.1)
(-5.5,0) -- (-3.5,0);

\filldraw [fill=gray!20,draw=black]
(-5.5,-2) arc (180:0:1.1)
(-5.5,-2) -- (-3.5,-2);

\draw[blue](-5.5,-2) -- (-3.5,-2);
\draw[blue](-5.5,0) -- (-3.5,0) node [pos=0.5,above] {$T$};
\end{tikzpicture}
\end{center}
where on $x < 0$
\begin{alignat*}{1}
a_1 &= c^0s + c^1s(s-s^0) -2s^0Dr \\
a_2 &= s^0\left( -Dr^2 + c_0 r\right)  \geq 0 \\
b_1 &= c^0 - 2dr \\
b_2 &= A - f'(\psi^0) + c_0r - dr^2 \geq 0 
\end{alignat*}
and where 
$$\varphi(c^1,s,v) =  -(c^0+c^1s)\partial_x\psi^0 + D\partial_{xx}\psi^0 - (c^0+c^1s)(s-s^0)\partial_x v + D(s-s^0)\partial_{xx}v$$

%and on $x > 1$
%\begin{alignat*}{1}
%a_2 = s^0(-D\partial_{xx}w  + c^0\partial_x w)  \geq 0 \\
%b_2 = A - f'(\psi^0) -d\partial_{xx}w  + c^0\partial_x w  \geq 0
%\end{alignat*}
%by choice of $r$ and the concavity and monotonicity of $w$ on $x>1$.
%$$\spadesuit = kr \text{ on } x<0\text{, }0\text{ on }x > 1$$
%$$\diamondsuit = c^0-2dr\text{ on }x<0\text{, }c^0\text{ on }x > 1$$ 
%$$\heartsuit = A - f'(\psi^0) + c^0r -dr^2\text{ on }x<0\text{, }A - f'(\psi^0)\text{ on }x > 1$$ 
%all being smooth functions and  where $r=\frac{c_{min}}{d}$, $\varphi(s,c^1,v) = -(c^0+c^1s)\partial_x\psi^0 - (c^0+c^
%1s)(s-s^0)\partial_x v$ and $k(s,c^1) = c^0s^0 + (c^0+c^1s)(s-s^0)$.
Using Schauder estimates up to the boundary for the Wentzell problem, we see that provided a global $L^\infty$ estimate for $w_1u$, $w_1u$ is in $\mathcal C^{2,\alpha}(B\cup T)$ with constant independent of the position of the closed half balls $B\cup T$ depicted on the diagram above. Since we can cover all $\overline{\Omega_L}$ with translations of $B\cup T$, this gives $w_1u \in \mathcal C^{2,\alpha}(\overline{\Omega_L})$. This weighted $L^\infty$ global estimate is the object of the next lemma. It simply relies on the maximum principle.	
\end{proof}

\begin{lem}
\label{weightedlinf}

Let $u=\tilde{\psi}(s,c^1,v)$. There exists two constants $K' < 0, K > 0$ such that
$$K' \leq w_1u \leq K$$
\end{lem}
\begin{proof}
We already have $u \leq Kw$ for $K > \text{max}(0,\text{sup } u)$ on $\Omega_L^+$. We now want to show that this is also (eventually with a larger constant) true in $\Omega_L^-$ by using the maximum principle.
Suppose there exists a point where $u > Kw$. That means that $Kw-u$ reaches a negative minimum somewhere in $\overline{\Omega_L^-}$ or tends to a negative infimum as $x \to -\infty$. First, let us see that this minimum cannot be reached. $Kw-u$ satisfies :

\begin{center}
\begin{tikzpicture}
\draw (-6,0) -- (6,0) node[pos=0.5,above] {$(d\partial_y + a_1(c^1,s)\partial_x)(Kw-u) = a_1(c^1,s)rKw - \varphi$};

\node at (0,-1) {$(\mathcal L + A)(Kw-u) = (-dr^2 + c^0r + A -f'(\psi^0))Kw > 0$};

\draw (-6,-2) -- (6,-2) node[pos=0.5,below] {$\partial_y (Kw-u) = 0$};

\end{tikzpicture}
\end{center}
In order to conclude to contradiction thanks to the Hopf lemma, we only need to ensure $k(c^1,s)rKw - \varphi > 0$, i.e. $a_1(c^1,s)rK > \text{sup}w_1\varphi$. Now observe that thanks to \eqref{r} and since $s\in[s^0,s^0+\delta]$ then provided 
\begin{equation}
\label{condc1}
c^1 > -\displaystyle\frac{s^0 c_{min}}{2\delta (s^0+ \delta) }
\end{equation}
we have $a_1 > \dfrac{s^0c_{min}}{2}$ so that $K>\text{max}\left(0,\dfrac{2\text{sup}(w_1\varphi)}{rs^0c_{min}}\right)$ suffices and in the end we have the desired result with $$K = \text{max}\left( \text{max}\left(0,\text{sup}(u)\right),\text{max}\left(0,\frac{2\text{sup}(w_1\varphi)}{rs^0c_{min}}\right)\right)$$ 
From now on, we assume condition \eqref{condc1} and we will see that this is not restrictive.

Now if the minimum is obtained at infinity, let us denote $(x_n, y_n)$ a minimizing sequence. Since $y_n$ is bounded we can extract a subsequence that converges to $y_\infty\in[-L,0]$. Let us set 
\begin{equation}
\label{translates}
(Kw-u)^n(x,y) := (Kw-u)(x+x_n,y+y_\infty)
\end{equation}
 We have two subcases :
\begin{enumerate}[i)]
\item $y_\infty \in ]-L,0[$. Then \eqref{translates} defines a sequence of uniformly bounded functions in some small ball $B$ in the interior of $\Omega_L$. By standard elliptic estimates, we can extract from it a subsequence that converges in $\mathcal C^2(B)$ to some $(Kw-u)^\infty$ that satisfies $(-d\Delta + c^0\partial_x + A)(Kw-u)^\infty \geq 0$ in $B$ but reaches its negative infimum $m < 0$ inside $B$ : as a consequence, $(Kw-u)^\infty \equiv m$ in $B$, but this is impossible since $Am < 0$.
\item $y_\infty = 0$ or $-L$ : the exact same analysis applies, replacing the ball $B$ by a half-ball $B_\pm$ supported on $y=0$ or $y=-L$ and using elliptic estimates up to the boundary, and Hopf's lemma.
\end{enumerate}

For the other bound, we proceed in the same way by looking at $u-K'w$ with $K' < \text{min}(0,\frac{\text{inf}u}{2})$ and using the existence of $\text{inf}(w\varphi)$, we get $$K' = \text{min}\left(\text{min}\left(0,\frac{\text{inf}(u)}{2}\right), \text{min}\left(0,\frac{2\inf(w_1\varphi)}{rs^0c_{min}}\right)\right)$$ that works.
\end{proof}

Thanks to this auxiliary function, we are now left with the following equivalent problem, on $v$ : 

\begin{center}
\begin{tikzpicture}
\draw (-6,0) -- (6,0) node[pos=0.5,above] {$\mathcal W v = 0$};

\node at (0,-1) {$\mathcal Lv + c^1\partial_x\psi^0= R(s-s^0,c^1,v) - \mathcal L\tilde\psi(s,c^1,v)$};

\draw (-6,-2) -- (6,-2) node[pos=0.5,below] {$\partial_y v = 0$};

\end{tikzpicture}
\end{center}
Calling $\mathcal P = <e^*,\cdot>\partial_x\psi^0$ and $\mathcal Q = Id - \mathcal P$ the projections onto $Y_2$ and $R(\mathcal L)$ we are now able to apply these projections onto the equation to get a set of two equations that are equivalent to this one. Nonetheless, since $\tilde\psi$ on the boundary $y=0$ depends on $c^1$ even when $s=s^0$, we should be careful and try to make this dependence explicit. For this, we need to have an explicit representation of $e^*$ to be able to compute the projections. This technical difficulty only comes from the fact that the unknown $c$ appears in the boundary condition of \eqref{frontswentzell}.

Thanks to the smoothness and decay properties of $e^*, v$ and $\partial_x \psi_0$, all the integration by parts make sense and we find

$$\int_{\Omega_L} e^*\mathcal Lv = \int_{y=0}(v d\partial_y e^* - e^* d\partial_y v) = 0$$

\begin{equation*}
\begin{split}
\int_{\Omega_L} e^*\mathcal L\tilde{\psi} =& \int_{y=0} e^* \left(\left(c^0+c^1s\right)\partial_x \psi^0 - D\partial_{xx}\psi^0 \right) \\
&+ (s-s^0)\int_{y=0} e^* \left(\left(c^0+c^1s\right)\partial_x(\tilde\psi + v) - D\partial_{xx}(\tilde\psi + v) \right)
\end{split}
\end{equation*}
and we get the first equation \footnote{where $\tilde\psi$ means $\tilde\psi(s,c^1,v) $ and $R$ means $R(s-s^0,c^1,v+\tilde\psi(s,c^1,v))$}:
\begin{equation}
\label{c1}
\begin{split}
c^1\left(1 + s\int_{y=0}e^*\partial_x\psi^0\right) =&  -\int_{y=0}e^*\left(c^0\partial_x\psi^0 - D\partial_{xx}\psi^0\right)\\ &+\int_{\Omega_L} e^*R \\
& - (s-s^0)\int_{y=0}e^*\left(\left(c^0+c^1s\right)\partial_x(\tilde\psi + v) - D\partial_{xx}(\tilde\psi + v) \right)
 \end{split}
\end{equation} 
The second equation should be seen as an equation on $v_R \in X_1$ with the decomposition $$v = v_N\partial_x\psi^0 + v_R$$ and $v_N\in\mathbb R$ being free : this is, of course, due to the $x$-translation invariance of \eqref{frontswentzell}. From now on, we fix $v_N \in \mathbb R$.

\begin{equation}
\label{vR}
\begin{split}
\mathcal Lv_R =&\ R - \left(\int_{\Omega_L} e^*R\right)\partial_x\psi^0 - \mathcal L\tilde\psi    \\  
&+(c^0+c^1s)\left(\int_{y=0}e^*\partial_x\psi^0 +(s-s^0)\int_{y=0}e^*\partial_x(\tilde{\psi}+v)\right)\partial_x\psi^0
 \end{split}
\end{equation} 

The system of equations  \eqref{c1}, \eqref{vR} is non-linear and coupled but in the case $s=s^0$ it is much simpler. It becomes  
 \begin{alignat}{1}
c^1\left(1 + s^0\int_{y=0}e^*\partial_x\psi^0\right) & =  -\int_{y=0}e^*\left(c^0\partial_x\psi^0 - D\partial_{xx}\psi^0\right) \label{c1s0} \\
 \mathcal Lv_R &= -\mathcal L\tilde{\psi} + (c^0+c^1s^0)\left(\int_{y=0}e^*\partial_x\psi^0\right)\partial_x\psi^0 \label{vrs0} 
 \end{alignat} which has clearly a unique solution : since $\int_{-\infty}^{+\infty} (e^*\partial_x\psi^0)(x,0)dx > 0$, \eqref{c1s0} has a solution $c^1_*$ that satisfies condition \eqref{condc1} provided $\delta$ is small enough. \eqref{vrs0} is automatically uniquely solvable with a solution $v_R^*$ since its right hand side lies in $R(\mathcal L)$ and does not depend on $v$.

Now for $s > s_0$ we said that this system was non-linear and coupled, but this is when the implicit function theorem does all the work. Since $X_1$ is closed in $X$ and $\mathcal L$ is Fredholm so image-closed, we have the right Banach setting to apply it. We may see this system of equations as $F(s,c^1,v_N,v_R) = 0$ with 
$$F : \left[s_0,s^0 + \delta\right]\times \left[-\frac{s^0c_{min}}{2\delta(s^0+\delta)},+\infty\right[ \times X_1 \to \mathbb R \times R(\mathcal L)$$ and $\delta > 0$ small enough so that condition \eqref{condc1} is satisfied, associating to its parameters the equations $\eqref{c1},\eqref{vR}$ in this order. $F$ is a $\mathcal C^1$ function because it consists in affine bounded operators composed with usual and $\mathcal C^1$ functions. Moreover, we can compute the differential of $F$ at $(s^0,c^1_*,v_R^*)$ with respect to $(c^1,v_R)$. In matrix representation, it is 

$$\begin{pmatrix}
  1+s^0\int_{x=0}e^* \partial_x\psi^0 & 0 \\
  * & \mathcal L
 \end{pmatrix}$$
which is invertible since $1+s^0\int_{x=0}e^* \partial_x\psi^0 > 0$, and $\mathcal L$ is invertible on $X_1$. That being, the implicit function theorem says that there exists $\delta' > 0$ and a neighbourhood $\mathcal V$ of $(c^1_*,v_R^*)$ such that for each $s\in [s^0,s^0+\delta'[$, the system of equations has a unique solution $(c^1_s,v_R^s)\in \mathcal V$. Then we can construct back $\psi$ from $c^1_s,v_R^s,v_N$ and it will clearly satisfy the original equation. The left limit condition for it is obtained directly because of the structure of $X$. The only thing left to show is that the right limit condition holds. This is the case provided $\delta$ is taken small enough, and it is the object of the next proposition. 
\begin{rmq}
Note that this is valid for every $v_N \in \mathbb R$, which will provide us with a whole $1$-dimensional manifold of solutions in the end. Of course, thanks to Proposition \ref{uniqueness} all of these solutions will be $x$-translates of each other.
\end{rmq}

\begin{prop}
\label{rightlimitcd}
Let $$c = c^0 + (s-s^0)c^1,\  \psi = \psi^0 + (s-s^0)\psi^1$$ If $\delta > 0$ is small enough, we have uniformly in $y$ : $$\lim_{x\to+\infty} \psi(x,y) = 1$$ 
\end{prop}

\begin{proof}
First, we show that $\psi < 1$, by contradiction. We know that $\psi \in \mathcal C^{2,\alpha}_w$ is bounded. Suppose there exists a point where $\psi > 1$. Then either $\psi - 1$ reaches a positive maximum somewhere, or it tends to a positive maximum as $x \to \infty$. These two cases are both impossible, because of respectively the argument given in theorem \ref{bounds} and the compactness argument given in the proof of Lemma \ref{weightedlinf} (take $B$ or $B_\pm$ small enough so that $f(\psi) < 0$ on it). So $\psi \leq 1$ and the strong maximum principle and Hopf's lemma and the fact that $\psi$ cannot be constant give $$\psi < 1$$

Now we fix $\varepsilon > 0$. For $a$ large enough we have $\psi^0 > 1-\frac{\varepsilon}{2}$ on $x \geq a$. Moreover, we can take $\delta$ small enough such that $|(s-s^0)\psi^1|_\infty < \frac{\varepsilon}{2}$, what gives $1-\varepsilon < \psi < 1$ for $x \geq a$. We assert that this property suffices to have $\psi \to 1$ for $\delta$ small enough, and we will show that by a maximum principle argument using an exponential solution to the right for the linearised problem near $1$.

On $x \geq a$, by Taylor's formula applied on $f$, we have $$-d\Delta (1-\psi) + c\partial_x (1-\psi) = f'(1)(1-\psi) + o(\varepsilon)$$ So by choosing $\varepsilon > 0$ small enough, we have

$$ L_1(1-\psi) := -d\Delta (1-\psi) + c\partial_x (1-\psi) - \frac{1}{2} f'(1)(1-\psi) \leq 0$$

We now look for a positive solution $p$ of $L_1p = 0$ endowed with the boundary condition of \eqref{frontswentzell} that has exponential decay as $x\to +\infty$, for comparison purposes. Unlike the proof of Proposition \ref{cmax}, we cannot expect a supersolution with the form $p(x,y) = e^{- \gamma x}$ with $\gamma > 0$, since inequations
\begin{alignat*}{1}
- d\gamma^2 - c\gamma - 1/2 f'(1) \geq 0 \\
s(-D\gamma^2 - c \gamma) \geq 0
\end{alignat*}
cannot be solved simultaneously. This motivates the research for a $p(x,y) = e^{-\gamma x} \phi(y)$, $\phi > 0$. For $p$ to be a solution of $L_1p = 0$ endowed with the boundary condition of \eqref{frontswentzell}, the equations are 

\begin{equation}
\begin{cases}
- \phi'' + \left(-\frac{1}{2d}f'(1) - \gamma\left(\gamma + \frac{c}{d}\right)\right)\phi = 0 \\
d\phi'(0) - s(-D\gamma^2-c\gamma)\phi(0) = 0 \label{eqgamma} \\
\phi'(-L) = 0
\end{cases}
\end{equation}
Since $f'(1) < 0$, this can be solved by 
$$\phi(y) = \cosh\left(\beta(\gamma)\left(y+L\right)\right)$$
where $$\beta(\gamma) = \sqrt{-f'(1)/(2d)-\gamma(\gamma+c/d)}$$
and $0 < \gamma < \gamma_{lim} = \frac{\sqrt{c^2 - 2df'(1)}-c}{2d}$ solving
$$s(D\gamma^2 + c\gamma) = d\beta(\gamma)\tanh(\beta(\gamma)L)$$ 
as pictured in Figure \eqref{figuregamma}.

    \begin{figure}[h!]
    \label{figuregamma}
    \centering
    
    \includegraphics[width=0.5\linewidth]{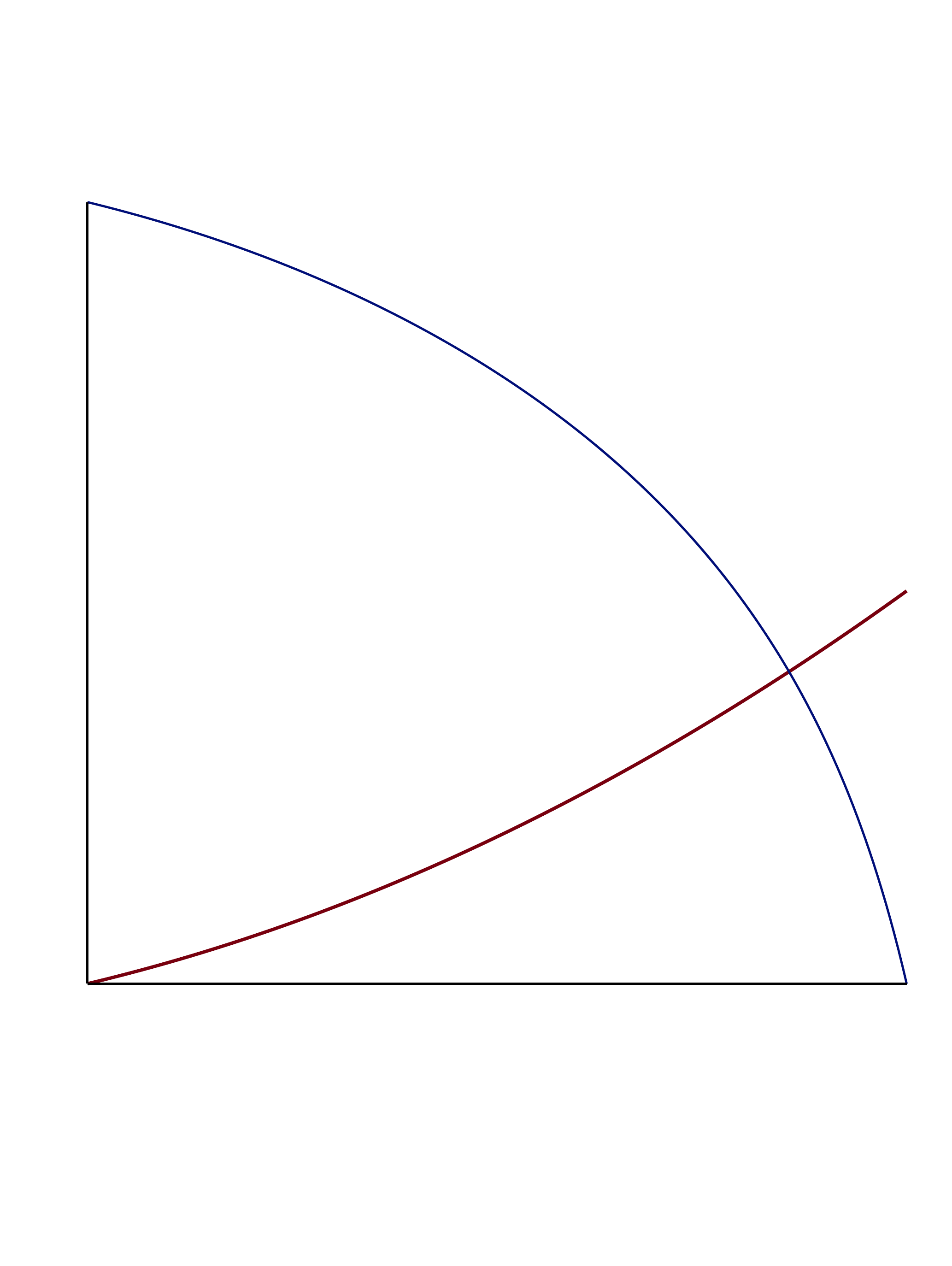}
    \rput(-7.4,2){$0$}
    \rput(-0.4,2){$\gamma_{lim}$}
    
    \rput(1.2,8.8){$d\beta(\gamma)\tanh(\beta(\gamma)L)$}
    \rput(4,6.6){$s(D\gamma^2 + c\gamma)$}
    \vspace{-70pt}
    
    \caption{Equation \eqref{eqgamma} on $\gamma$}
    \end{figure}  
Now chose $C > 0$ such that $1-\psi < Cp$ on $x=a$ and observe that $U = Cp - (1-\psi)$ solves on $x\geq a$
\begin{equation*}
  \begin{tikzpicture}
  \draw (-5.5,0) -- (7.5,0) node[pos=0.6,below] {\small{$d\partial_y U + s\left(-D\partial_{xx} U + c\partial_x U\right) = 0$}} ;

  \node at (1.5,-1.5) {$L_1 U \geq 0 $};

  \draw (-5.5,-3) -- (7.5,-3) node[pos=0.5,above] {\small{$-d\partial_y U = 0$}};
    
  \draw (-5.5,-3) -- (-5.5,0) node[pos=0.5,left]{$U > 0$};

  \end{tikzpicture}
\end{equation*}
Now suppose that there is a point where $U < 0$. Then either $U$ reaches a negative minimum or tends to a negative infimum $m < 0$ as $x\to +\infty$. The first case is impossible thanks to the strong maximum principle and Hopf's lemma. The second is impossible also thanks to the compactness argument already given in Proposition \ref{weightedlinf}, since $L_1(m) = -1/2 f'(1)m < 0$. As a consequence, for all $x\geq a$ :

$$ 0 < 1-\psi \leq Ce^{-\gamma x}\phi(y) \leq C\max(\phi) e^{-\gamma x}$$
which gives the desired result by sending $x\to +\infty$.

\end{proof}

This section is now finished and Theorem \ref{w} is proved. 
\begin{rmq}
Note that the above subsection does not apply exactly when $s^0=0$. Indeed, in this case the estimates up to the Wentzell boundary do not hold. Nonetheless, this situation is way simpler : just apply the standard estimates up to the Neumann boundary. We leave it to the reader to check that everything holds, condition \eqref{condc1} being replaced by $c_1 > -c^0/\delta$ and all the other computations being simpler (for instance, no information on $e^*$ is needed).
\end{rmq}

\subsection{Proof of lemma \ref{lemJMR}}
\label{propL}
\subsubsection*{Proof of the Fredholm property}
$\mathcal L$ is Fredholm of index $0$ as an operator $\mathcal C^{2,\alpha}_w \to \mathcal C^\alpha_w$ if and only if $\tilde{\mathcal L}u := \frac{1}{w} \mathcal L(wu)$ defines a Fredholm operator of index $0$ as an operator $\mathcal C^{2,\alpha} \to \mathcal C^\alpha$, endowed with the boundary condition $\partial_yu = 0$ on $y=-L$ and $d\partial_y u + \frac{1}{w} c^0s^0 \partial_x(wu) = 0$ on $y=0$.

We do not have any closed formula for the coefficients of $\tilde{\mathcal L}$, but we know that
$$\tilde{\mathcal L}u = -d\Delta u + (c^0-2dr)\partial_x u + (c^0r - dr^2 - f'(\psi^0))u \text{ on } x < 0, \tilde{\mathcal L} = \mathcal L \text{ on } x>1$$
Moreover the $0$-order coefficient of $\tilde{\mathcal L}$ is $c^0r - dr^2 > 0$ on $x<0$, and tends to $-f'(1) > 0$ uniformly in $y$ as $x\to\infty$ ; thus it is greater than some positive constant, away from a compact set : this indicates a decomposition invertible + compact for $\mathcal L$.

The boundary condition $\tilde{\mathcal L}$ is endowed with is unchanged on $y=-L$ and is $$d\partial_y u + c^0s^0\partial_x u + (\cdots \geq 0 )u = 0$$ on $y=0$, thanks to the definition of $r$ and the properties of $w$ asked in Definition \ref{defw}.

So call $\gamma(x)$ a positive function that smoothly connects $c^0r - dr^2 > 0$ on $x<0$ with $-f'(1)$ on $x>1$ such that $\gamma \geq \min(c^0r - dr^2 > 0,-f(1)) := \gamma_0 > 0$. We now call $\tilde{\mathcal T}$ the operator $\tilde{\mathcal L}$ with its $0$-order coefficient replaced by $\gamma(x)$, and we want to show that $\tilde{\mathcal T}$ is invertible, and that $\tilde{\mathcal S} := \tilde{\mathcal L} - \tilde{\mathcal T}$ satisfies $\tilde{\mathcal S}\tilde{\mathcal T}^{-1}$ is compact on $\mathcal C^\alpha$ and $\tilde{\mathcal T}^{-1}\tilde{\mathcal S}$ on $\mathcal C^{2,\alpha}$, in order to have $$\tilde{\mathcal L} = (Id + \tilde{\mathcal S}\tilde{\mathcal T}^{-1})\tilde{\mathcal T} = \tilde{\mathcal T}(Id + \tilde{\mathcal T}^{-1}\tilde{\mathcal S})$$ which is the Fredholm property with index $0$ we want.

First suppose that $\tilde{\mathcal T}$ is indeed invertible : then the compactness of the perturbation is easy to obtain. Indeed, $\tilde{\mathcal S}$ is no more than the multiplication by a function that is $\equiv 0$ on $x\leq 0$ and that tends uniformly in $y$ to $0$ as $x\to\infty$. So, taking $(u_n)$ a bounded sequence in $\mathcal C^\alpha(\Omega_L)$, we have that $(\tilde{\mathcal T}^{-1} u_n)$ is bounded in $\mathcal C^{2,\alpha}$, so by applying a chain of Ascoli theorems and the process of diagonal extraction we can extract from $(\tilde{\mathcal T}^{-1} u_n)$ a sequence we note $(v_n)$ that converges in $\mathcal C^2_{loc}$ to $v$. We now want to extract from $(\tilde{\mathcal S} v_n)$ a sequence that converges in $\mathcal C^\alpha$. But this is easy since $\tilde{\mathcal S} v_n = 0$ on $x<0$ and $\tilde{\mathcal S} v_n \to 0$ uniformly in $y$ as $x\to\infty$, so in fact the $\mathcal C^2_{loc}$ convergence of $v_n$ suffices to have $\tilde{\mathcal S} v_n \to \tilde{\mathcal S}v$ in whole 
$\mathcal C^2$, so in $\mathcal C^\alpha$. For $\tilde{\mathcal T}^{-1}\tilde{\mathcal S}$ on $\mathcal C^{2,\alpha}$ we apply the same argument : we just have to see that $\tilde{\mathcal T}^{-1}(\tilde{\mathcal S}(u_n))$ is bounded in $\mathcal C^{4,\alpha}$ since $u_n$ is bounded in $\mathcal C^{2,\alpha}$. Then we extract from it something that converges in $\mathcal C^3_{loc}$, but in fact, in whole $\mathcal C^3$ so in $\mathcal C^{2,\alpha}$.

It remains to show that $\tilde{\mathcal T} : \mathcal C^{2,\alpha} \to \mathcal C^\alpha$ is indeed invertible, that is, to show that the following problem is uniquely solvable
\begin{center}
\begin{tikzpicture}
\draw (-6,0) -- (6,0) node[pos=0.5,above] {$d\partial_y u +
(\cdots)\partial_x u + (\cdots \geq 0) u = 0$};

\node at (0,-1) {$\tilde{\mathcal T} u = f \in \mathcal C^\alpha(\Omega_L)$};

\draw (-6,-2) -- (6,-2) node[pos=0.5,below] {$\partial_y u = 0$};

\end{tikzpicture}
\end{center}
but this is the case, since the $0$-order coefficient of $\tilde{\mathcal T}$ is $>0$ (see theorem 6.31 in \cite{GT} and more precisely the remark at the end of its proof).

\subsubsection*{Computation of the kernel}
Suppose $\mathcal Lu = 0$. We will show that $$P := \{\Lambda \in \mathbb R \ | \forall \lambda < \Lambda, \ u > \lambda \partial_x\psi^0 \}$$ has a supremum $\lambda_0$, and that $u = \lambda_0\partial_x\psi^0$. First, we show that this set is non-void : for every truncated (compact) rectangle $K$, we can find $\lambda\in \mathbb R$ such that $u > \lambda \partial_x\psi^0$ on $K$.  Now just chose $K$ big enough such that outside $K$ we have $f'(\psi^0) \leq 0$, so we have the strong maximum principle, and since $\mathcal L(u-\lambda\partial_x\psi^0) = 0$, the comparison $u - \lambda \partial_x\psi^0 > 0$ is inherited in all $\Omega_L$ \footnote{a point where it is $\leq 0$ at the left of $K$ means that a non-positive minimum is reached at the left of $K$, which is impossible ; the right of $K$ is treated in the same way but with the compactness argument given in Lemma \ref{weightedlinf} since we do not know a priori that $u - \lambda \partial_x\psi^0 \to 0$ as $x\to\infty$ even if it is the case.}. Now $P$ being non-void and trivially bounded from above, it has the supremum we announced. By continuity, $u - \lambda_0\partial_x\psi^0 \geq 0$, and moreover we have $\mathcal L(u - \lambda_0\partial_x\psi^0) = 0$. Now suppose $u - \lambda_0\partial_x\psi^0 \not\equiv 0$ by contradiction : because of the strong maximum principle, we have $u - \lambda_0\partial_x\psi^0 > 0$, and again on any truncated rectangle $K$ we can find $\varepsilon > 0$ small enough such that $u > (\lambda_0 + \varepsilon) \partial_x\psi^0$ on $K$, and choosing $K$ large enough and proceeding as above, we have a contradiction regarding the maximality of $\lambda_0$.

Now, suppose $\mathcal L^2 u = 0$. Then $\mathcal Lu = \alpha\partial_x\psi^0$ for some $\alpha \in \mathbb R$. We suppose $\alpha \neq 0$ and we will obtain a contradiction. By linearity we can suppose $\alpha = 1$, i.e. $\mathcal Lu = \partial_x\psi^0 > 0$. Now, the fact that for every $\lambda\in\mathbb R$, $\mathcal L(u-\lambda\partial_x\psi^0) = \partial_x\psi^0$ is positive, enables to do the exact same proof as above to have a contradiction too (we will necessarily have $u > \lambda_0 \partial_x \psi^0$ and the contradiction, since $\mathcal Lu \neq 0$).

\subsubsection*{Properties of $e^*$}
Finally, let $e^*$ generate the kernel of the adjoint of $\mathcal L$. Let us normalise $e^*$ by the condition $<e^*,\partial_x\psi_0> = 1$, and show that $e^*$ is a positive measure. Similarly to \cite{JMR92}\footnote{where the author treats this exact problem with a Neumann condition instead of a Wentzell one, but this does not change his proof.} we infer that $\mathcal L$ is sectorial on $BW_0 := \{u \in UC_0(\Omega_L)  \mid w_1u \in UC_0(\Omega_L)\}$ and that $0$ is the bottom of its spectrum. As a consequence we have the following realisation of $e^*$ on $BW_0$ :
$$\forall u_0\in BW_0\text{, }\lim_{t\to+\infty} e^{-tL}u_0 = <e^*,u_0>\partial_x\psi_0$$
Indeed, decomposing $u_0$ on $N(\mathcal L) \subset BW_0$ and its orthogonal complement we get $e^{-tL}u_0 = e^{-tL}(<e^*,u_0>\partial_x\psi_0 + b_0)$, the first term being constantly $<e^*,u_0>\partial_x\psi_0$ and the second one decaying exponentially fast to zero as $t\to+\infty$. Knowing that $\partial_x\psi_0 > 0$ and applying this on every non-negative $u_0$ in $\mathcal D(\Omega_L) \subset BW_0$, since non-negativity is preserved over time for $e^{-tL}u_0$, we get that $e^*$ is a non-negative distribution, that is a positive measure.

Moreover, $e^*$ satisfies $\mathcal L^* e^* = 0$ in the sense of distributions along with its dual boundary condition, which is an hypoelliptic problem (see \cite{Cat92} Thm 4.2 or \cite{Pan86} Thm 3)(iii)). As a consequence, $e^*$ is a smooth non-negative function up to the boundary of $\Omega_L$. 

Then, the strong maximum principle gives $e^* > 0$. Finally, using the weak Harnack inequality up to the boundary of \cite{Luo93} and the classical subsolution estimate up to the boundary of \cite{GT}, Theorem 9.20, we obtain a full Harnack inequality up to the boundary for $e^*$. Using it in $x$ large (where $-f'(\psi_0) > 0)$) on half-balls touching the boundaries we get that $e^*$ is bounded : indeed, if its supremum were blowing up, its infimum would also : but this is impossible since $e^*$ is integrable on $x>0$ (note that $w \in X$).

The same argument for $-x$ large gives that $e^*$ has at most a $Ce^{-rx}$ growth.

\section{Continuation from small $\varepsilon > 0$ to $\varepsilon = 1$}
\label{epsto1}
First we avoid the singularity near $\varepsilon = 0$ : it will be studied later since it deals with a very unusual boundary condition. 
Let us set for $\varepsilon_0 > 0$, $$P_{\varepsilon_0} = \{\varepsilon \in [\varepsilon_0,1] \quad | \quad \eqref{Seps}\text{ has a solution}\}$$
We now adapt the proofs of the previous section, following the same steps. The main differences are technical : all the computations are adapted easily, the counterpart of the regularity result in Proposition \ref{schauderwentzell} has no technical difficulty any more, but the weight function in Section \ref{pepsopen} changes a bit. For technical reasons we had to chose $e^{rx}$ everywhere, so we will need to be careful about the boundedness of solutions.

\subsection{$P_{\varepsilon_0}$ is closed}
This subsection follows exactly subsection \ref{pwclosed}. Consider a sequence $\varepsilon_n \to \varepsilon_\infty \in [\varepsilon_0,1]$ and call $(c_n, \phi_n, \psi_n)$ the associated sequence of solutions of \eqref{Seps} normalised in translation by
\begin{equation}
	\label{condsys}
 \max_{x \leq 0, y\in[-L,0]}\ (\mu u_\varepsilon(x), v_\varepsilon(x,y)) = \theta
\end{equation}

Thanks to Propositions \ref{sectioncmax} and \ref{cmin} we can extract from $c_n$ a subsequence such that 
\begin{equation}
\lim_{n\to+\infty} c_n = c_\infty > 0
\end{equation}
We now state a regularity result which is the counterpart of Proposition \ref{schauderwentzell} in the case of \eqref{Seps} :

\begin{prop}
There exists $\alpha > 0$ and constants $C_{Sch1,2}=C(D,d,c_{max},\text{Lip} f,L,\mu)$ such that for all $n \geq 0$
\label{schaudersystem}
\begin{alignat*}{1}
|\psi_n|_{\mathcal C^{2,\alpha}(\Omega_L)} &\leq C_{Sch1}\left(|\psi_n|_{L^\infty(\Omega_L)} + |\mu\phi_n|_{L^\infty(\Omega_L)} \right) \leq 2C_{Sch1} \\
|\mu\phi_n|_{\mathcal C^{2,\alpha}(\mathbb R)} &\leq C_{Sch2}\left(|\psi_n|_{L^\infty(\Omega_L)} + |\mu\phi_n|_{L^\infty(\Omega_L)} \right) \leq 2C_{Sch2} 
\end{alignat*}

\end{prop}

\begin{proof}
We adapt the proof of Proposition \ref{schauderwentzell}. By classical ODE theory (use Fourier transform or the variation of constants), there exists $C_{ode} = C(D,\mu,c_{max})$ such that
$$|\mu\phi_n|_{\mathcal C^{1,\alpha}} \leq |\mu\phi_n|_{W^{2,\infty}} \leq C_{ode}|\mu\phi_n|_\infty \leq C_{ode}$$
Seeing the right-hand side $f(\psi_n)$ in \eqref{Seps} as $-\frac{f(\psi_n)}{\psi_n}\psi_n$ in the left-hand side, which yields a bounded $0$-order term since $f$ is Lipschitz, we can use the Hölder continuity estimate up to the mixed boundary of \cite{Lieb} and iterate with the classical Schauder estimate up to the Robin boundary (see \cite{GT}, Lemma 6.29) so that on half-balls $B_-$ supported on $y=0$ on a segment $T$ :
$$|\psi_n|_{\mathcal C^{2,\alpha}(B_-)} \leq C_R\left( |\psi_n|_{L^\infty(2B_-)} + C_{ode}|\mu\phi_n|_\infty \right)$$
for some constant $C_R$. Finally we obtain the desired result by plugging the above estimate in standard Schauder estimates for $\phi$ : 
$$|\phi_n|_{\mathcal C^{2,\alpha}(T)} \leq C_{Sch}\left( |\psi_n|_{L^\infty(2T)} + C_{R}\left( |\psi_n|_{L^\infty(2T)} + C_{ode}|\mu\phi|_\infty \right)\right)$$

As before, we obtain the global estimate by covering $\mathbb R \times \Omega_L$ with such $T$ and $B_-$ using that the above estimate holds independently of the position of $B_-$, and other half-balls where standard Schauder estimates up to the Neumann boundary hold.
\end{proof}

Thanks to the previous estimate, as before we extract from $(\phi_n,\psi_n)$ a subsequence still denoted $(\phi_n,\psi_n)$ that converges in $\mathcal C^2_{loc}$ to $\phi_\infty, \psi_\infty$ satisfying $(S_{\varepsilon_\infty})$ except the limiting conditions. We now conclude just as in Propositions \ref{expdecay} and \ref{rightlimit} :

\begin{prop}
$\mu\phi_\infty$ and $\psi_\infty$ satisfy uniformly in $y$
$$\lim_{x\to -\infty} \mu\phi_\infty(x), \psi_\infty(x,y) = 0$$ 
$$\lim_{x\to +\infty} \mu\phi_\infty(x), \psi_\infty(x,y) = 1$$ 
\end{prop}

\begin{proof}
For the left limit, just observe that thanks to condition \eqref{condsys}, \eqref{compgauche} still holds for both $\mu\phi_\varepsilon$ and $\psi_\varepsilon$.
For the right limit, the computations of Proposition \ref{rightlimit} still hold : the only difference is that the boundary term $$\frac{s_\infty}{\mu} \int_0^M \left(-D\partial_{xx}\psi^\infty(x,0) + c_\infty \partial_x \psi^\infty(x,0)\right) dx$$ should be replaced here by $$\int_0^M \left(-D\partial_{xx}\phi^\infty(x) + c_\infty \partial_x\phi^\infty(x) \right) dx$$ which is treated in the exact same way.
\end{proof}

\subsection{$P_{\varepsilon_0}$ is open}
\label{pepsopen}
To simplify the notations, we note $M = 1/\varepsilon$ and we search around a solution $(c^0,\phi_0,\psi_0)$ for $M=M_0$, a solution $c = c^0 + (M-M_0)c^1, \phi = \phi_0 + (M-M_0)\phi_1, \psi = \psi_0 + (M-M_0) \psi_1$. The equations on $c^1,\phi_1,\psi_1$ are

\begin{equation}
\begin{tikzpicture}
\draw (-7,0) -- (7,0) node[pos=0.5,below] {\small{$E(\phi^1,\psi^1) = (M-M_0)(\mu\phi^1 - \psi_1) + \mu\phi_0 - \psi_0 $}} node[pos=0.5,above] {$l(\phi^1,\psi^1) =  (M-M_0)(\phi_1 - \mu\psi_1) - c^1\phi_0' - (M-M_0)c^1\phi_1' + \psi_0 - \mu\phi_0$} ;

\node at (0,-1.5) {$\mathcal L\psi_1 + c^1\partial_x\psi_0 = R(M-M_0,c^1,\psi_1)$};

\draw (-7,-3) -- (7,-3) node[pos=0.5,above] {\small{$\partial_y\psi = 0$}};

\end{tikzpicture}
\end{equation}
where \begin{alignat*}{1}
&l(\phi,\psi) = -D\phi'' + c^0\phi' - M_0(\phi - \mu\psi) \\ 
&\mathcal L\psi = -d\Delta \psi + c^0\partial_x  \psi - f'(\psi^0)\psi \\
&E(\phi,\psi) = d\partial_y \psi - M_0(\mu\phi - \psi) 
\end{alignat*}

The functional setting will be $r=\min(\frac{c_{min}}{D}, \frac{c_{min}}{d})$, $w(x)=e^{rx}$ on the whole real line (we will see later why we need to take the exponential everywhere instead of connecting it with a constant like before) $$X=\mathcal C^{2,\alpha}_w(\mathbb R) \times \mathcal C^{2,\alpha}_w(\Omega_L)$$ and we will work with the operator from $X$ to $Y = \mathcal C^{\alpha}_w(\mathbb R) \times \mathcal C^{\alpha}_w(\Omega_L)$ 

$$\mathscr{L}(\phi,\psi) := (l(\phi,\psi), \mathcal L\psi)$$
endowed with the exchange condition $E(\phi,\psi) = 0$ on $y=0$ and the Neumann condition $-d\partial_y\psi = 0$ on $y=-L$. 

Treating the boundary as usual in the system case, we can obtain the same properties as in Lemma \ref{lemJMR}, with this time $N(\mathscr L)$ generated by $(\phi_0',\partial_x\psi^0)$ :
\begin{lem}
\label{lemJMRsys}
$\mathscr L: X\to Y$ is a Fredholm operator of index $0$. As a consequence, the following decompositions hold 
$$X=N(\mathscr L)\oplus X_1$$
$$Y=R(\mathscr L)\oplus Y_2$$
where $X_1 \simeq R(\mathcal L)$  is a closed subspace of $X$ and $Y_2 \simeq N(\mathcal L)$. Moreover
$$N(\mathscr L) = N(\mathscr L^2) = \mathbb R (\phi_0',\partial_x{\psi_0})$$
\end{lem}
\begin{proof}
This proof is postponed in section \ref{propLsystem} to lighten this section. It relies on the same arguments as Lemma \ref{lemJMR}, up to the subtlety of the system case. These technicalities are the reason why we chose $w(x) = e^{rx}$ everywhere. Observe that the exponential growth of $w$ as $x\to+\infty$ adds a difficulty in proving the Fredholm property : we have to prove that the invertible part of $\mathscr L$ yields bounded solutions.
\end{proof}

In order to work with this fixed problem, we have to kill the non-homogeneities and the small terms, so as before we look for solutions with form $$\phi^1 = \tilde{\phi}(M,c^1,\phi,\psi) + \phi$$ $$\psi^1 = \tilde{\psi}(M,c^1,\phi,\psi) + \psi$$ where $\tilde{\phi}, \tilde{\psi}$ solves, for $A$ large enough,

\begin{equation}
\begin{tikzpicture}
\draw (-7.5,0) -- (7.5,0) node[pos=0.5,below] {\small{$d\partial_y\tilde\psi - M(\mu\tilde\phi - \tilde\psi)  = (M-M_0)(\mu\phi - \psi) + \mu\phi_0 - \psi_0 $}} node[pos=0.5,above] {$-D\tilde\phi'' + c_0\tilde\phi' - M(\tilde\psi - \mu\tilde\phi) + A\tilde{\phi} = (M-M_0)(\psi-\mu\phi) - c^1\phi_0' - (M-M_0)c^1\phi' + \psi_0 - \mu\phi_0$} ;

\node at (0,-1.5) {$(-d\Delta + c_0\partial_x - f'(\psi^0) + A)\tilde{\psi} = 0$};

\draw (-7.5,-3) -- (7.5,-3) node[pos=0.5,above] {\small{$\partial_y\psi = 0$}};

\end{tikzpicture}
\end{equation}

\begin{lem}
Such a function $(\tilde\phi, \tilde\psi)$ exists and satisfies $$(\tilde\phi,\tilde\psi)\in \mathcal C^1 \left(\mathbb R \times \mathbb R \times X; X\right)$$
\end{lem}

\begin{proof}
See Section \ref{propLsystem} for the solvability of this equation provided $A$ large enough. The fact that $\tilde\phi,\tilde\psi$ are not only $\mathcal C^{2,\alpha}$ but $\mathcal C^{2,\alpha}_w$ is shown just as before : thanks to the Schauder type estimate as in Proposition \ref{schaudersystem}, it suffices to show that $w_1\tilde\psi$ and $w_1\tilde\phi$ are bounded. For this, repeat the proof of Lemma \ref{weightedlinf} but treating the boundary as usual in the system case.
\end{proof}

Thus we are left with the following problem to solve in $c^1\in\mathbb R$, $(\phi,\psi) \in X$ : 
\begin{equation}
\label{syshomcont}
\mathscr L(\phi,\psi) + c^1(\phi_0',\partial_x\psi^0) = (r,R) - \mathscr L(\tilde\phi,\tilde{\psi})
\end{equation}
As before, applying the projection $\mathscr P$ onto $Y_2$ on \eqref{syshomcont} yields an equation on $c^1$, and applying $\mathscr Q = Id - \mathscr P$ yields an equation on the image part of the decomposition of  $(\phi,\psi) = \Lambda(\phi_0',\partial_x\psi^0) + (\phi_R,\psi_R) \in X$, $\Lambda \in \mathbb R$ being free in all this procedure.  The set of equation obtained is 
\begin{alignat}{1}
& c^1 = \mathscr P((r,R) - \mathscr L(\tilde\phi,\tilde{\psi})) \label{c1sys} \\
& \mathscr L(\phi_R,\psi_R) = \mathscr Q((r,R) - \mathscr L(\tilde\phi,\tilde{\psi}))\label{vRsys}
\end{alignat}

For $M>M_0$, the auxiliary functions depend on $c^1, \phi$ and $\psi$ so this system is non-linear and coupled, but at $M=M_0$, we have $r=R \equiv 0$, and the auxiliary functions depend only on $\phi_0,\psi_0$, so in this case the system, as before, can be solved step by step. Moreover, since here $c$ does not appear in the boundary condition, we do not need the duality argument of the previous section : \eqref{c1sys} is trivially solvable.

Finally, as before the differential of this system of equations with respect to $c_1, (\phi_R,\psi_R) \in \mathbb R \times X_1$  at $M=M^0$ and the corresponding solutions yields an isomorphism since $\mathscr L$ is invertible on $X_1$, and the implicit function theorem provides for $M$ close to $M^0$ a solution of \eqref{Seps} apart from the limiting conditions.

The right-limit condition is then obtained by an adaptation of the computations of Proposition \ref{rightlimitcd}. We wish to emphasise on the fact that even though $w$ has exponential growth as $x\to +\infty$, $\phi^1, \psi^1$ are indeed bounded, as highlighted in Lemma \ref{lemJMRsys}.

\begin{prop}
\label{rightlimitcdsys}
Let $$c = c^0 + (M-M^0)c^1,\ \phi = \phi_0 + (M-M^0)\phi^1,\ \psi = \psi_0 + (M-M^0)\psi^1$$ If $M - M^0$ is small enough, we have $$\lim_{x\to+\infty} \mu\phi(x), \psi(x,y) = 1$$ uniformly in $y$.
\end{prop}
\begin{proof}
By treating the upper boundary as usual in the system case, the arguments of Proposition \ref{rightlimitcd} hold. The only thing to check is the existence of another supersolution with exponential decrease in this case. We look for a solution of the type $(e^{-\gamma x}, e^{-\gamma x}h(y))$. The equations on $\gamma > 0, h > 0$ are 

\begin{alignat}{1}
- h'' + \left(-\frac{1}{2d}f'(1) - \gamma\left(\gamma + \frac{c}{d}\right)\right)h &= 0 \\
dh'(0) &= \mu - h(0) \label{eqA} \\
-h'(-L) &= 0 \\
-D\gamma^2 - c\gamma &= h(0) - \mu \label{eqgammasys}
\end{alignat}
Since $f'(1) < 0$, this can be solved by 
$$h(y) = A\cosh\left(\beta(\gamma)\left(y+L\right)\right)$$
where $$\beta(\gamma) = \sqrt{-f'(1)/(2d)-\gamma(\gamma+c/d)}$$
Moreover, equation \eqref{eqA} gives that 
$$ A = \dfrac{\mu}{d\beta(\gamma)\sinh(\beta(\gamma)L) + \cosh(\beta(\gamma)L)}$$
which, plugged in equation \eqref{eqgammasys} yields the equation on $\gamma$ : 
$$D\gamma^2 + c\gamma = \dfrac{\mu d \beta(\gamma)}{1 + \tanh(\beta(\gamma) L)}$$
which has a solution $0 < \gamma < \gamma_{lim}$ for the same reasons as in Proposition \ref{rightlimitcd}.
    \begin{figure}[h!]
    \label{figuregammasys}
    \centering
    \vspace{-30pt}

    \includegraphics[width=0.5\linewidth]{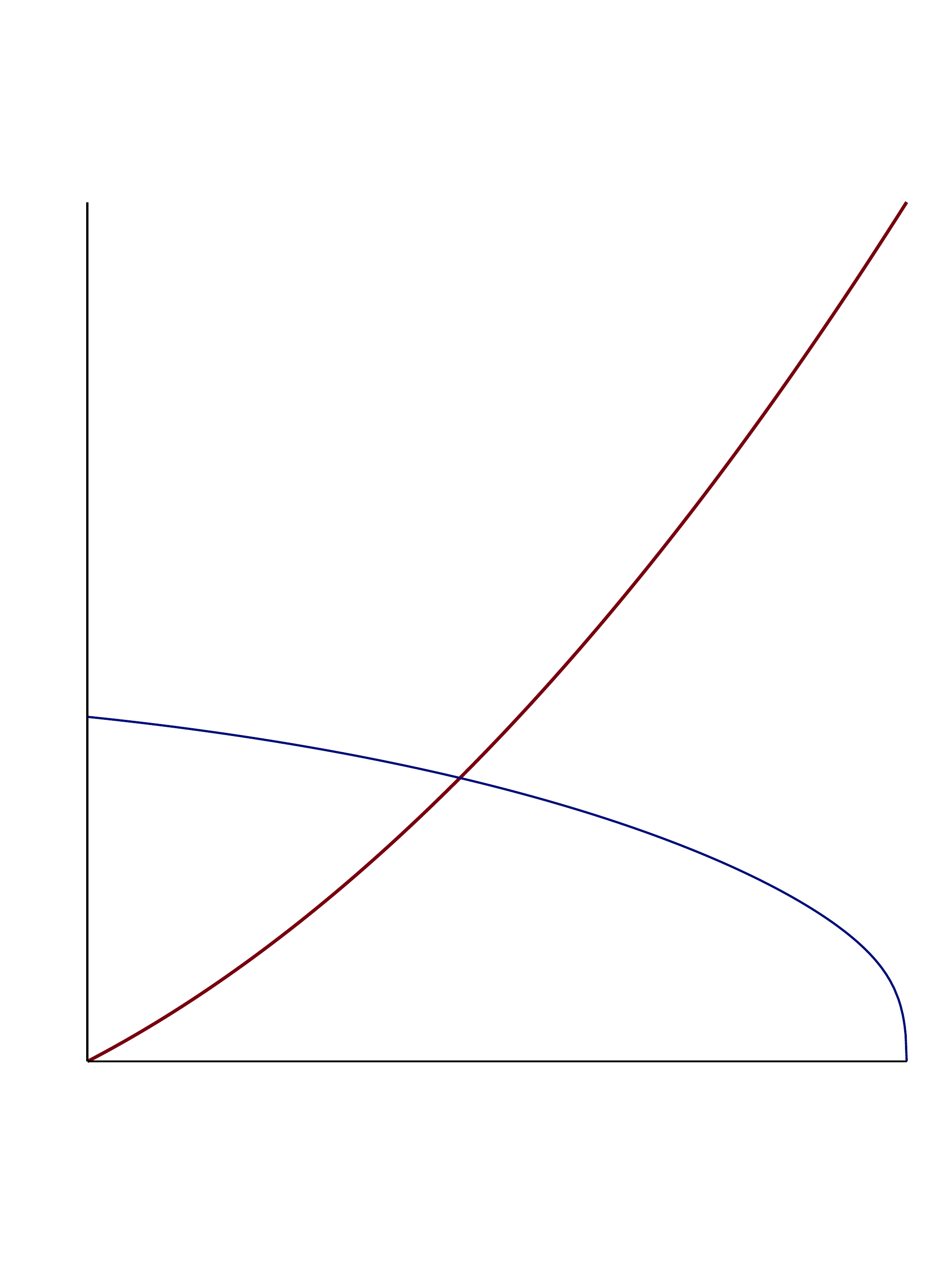}
    \rput(-7.6,1.4){$0$}
    \rput(-0.4,1.4){$\gamma_{lim}$}
    
    \rput(-1.3,5.5){$\frac{\mu d \beta(\gamma)}{1 + \tanh(\beta(\gamma) L)}$}
    \rput(2.5,6.6){$D\gamma^2 + c\gamma$}
    \vspace{-60pt}
    
    \caption{Equation \eqref{eqgammasys} on $\gamma$}
    \end{figure}
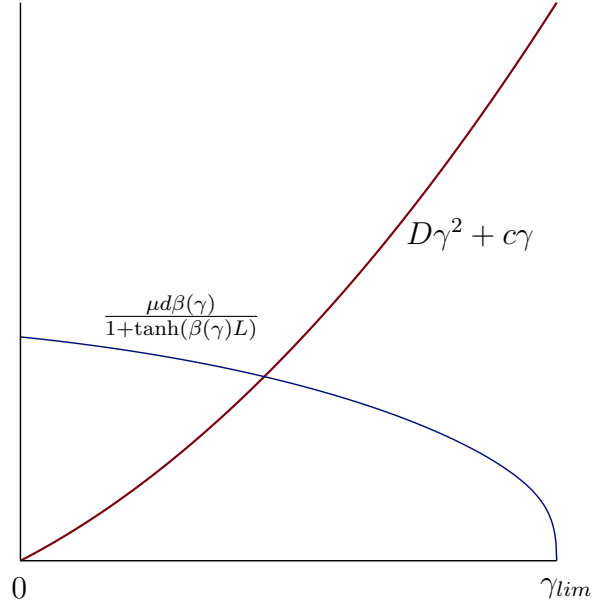

\end{proof}

\subsection{Proof of Lemma \ref{lemJMRsys}}
\label{propLsystem}
Throughout all this section, in order to simplify the notations, we have taken $M_0 = 1$ without loss of generality. In this section we show that $\mathscr L$ is Fredholm of index $0$ on $X$, and that $N(\mathscr L) = N(\mathscr L^2)$ is generated by $(\phi_0',\partial_x\psi_0)$.
The proof of the second property does not change : $(\phi_0',\partial_x\psi_0)$ is indeed a solution of the problem, and by treating the boundary condition as usual in the system case, the proof of lemma \ref{lemJMR} as in section \ref{propL} still holds. 
The proof of the Fredholm property on the other hand changes a bit, since we did not take the usual weight but the exponential weight on the whole real line.  This is because of the exchange condition : suppose we had take the usual weight, and did all the machinery $\tilde{\mathscr L} = \tilde{\mathscr T} + \tilde{\mathscr S}$. Then we would not be able to show that $\tilde{\mathscr T}$ is invertible. Indeed, suppose we want to solve $\tilde{\mathscr T}(\phi,\psi) = (g,h) \in \mathcal C^\alpha(\mathbb R)\times \mathcal C^{\alpha}(\Omega_L)$. In order to obtain that $\tilde{\mathscr T}$ is injective (and that its inverse is bounded if it exists), we want to control $(\phi,\psi)$ by the data $(g,h)$, by starting with the $L^\infty$ norm. So suppose $\psi$ reaches a maximum somewhere. Then if it is on the road and on $x > 1$, we have a problem. Indeed, the Hopf lemma only gives $\psi < \mu\phi$ and then looking the equation on the road gives nothing : that is why we want to pull a bit the $0$-order 
coefficient on the road, and that is why we have chosen $w(x) = e^{rx}$ everywhere, so that $\tilde{\mathscr T}(\phi,\psi) = (g,h)$ is no more than

\begin{equation}
\label{Ttilde}
\begin{tikzpicture}
\draw (-6,0) -- (6,0) node[pos=0.5,below] {\small{$d\partial_y\psi = \mu\phi -\psi$}} node[pos=0.5,above] {$-D\phi'' + (c_0-2Dr)\phi' + (\mu+\alpha_r)\phi - \psi = g$} ;

\node at (0,-1.5) {$(-d\Delta + (c_0-2dr)\partial_x +\gamma(x))\psi = h$};

\draw (-6,-3) -- (6,-3) node[pos=0.5,above] {\small{$\partial_y\psi = 0$}};

\end{tikzpicture}
\end{equation}
with $\alpha_r = -Dr^2 + c_0r > 0$. In this setting, a maximum point of $\psi$ reached on the road is no more a problem, we always have that $\psi < \frac{\mu}{\alpha_r}|g|_\infty$ in this case, and actually, in every case

$$ | \psi |_\infty \leq \frac{1}{\min \gamma} |h|_\infty + \frac{\mu}{\alpha_r}|g|_\infty $$

$$ |\phi |_\infty \leq \frac{|g|_\infty + |\psi|_\infty}{\mu + \alpha_r} \leq \frac{1}{\min \gamma(\mu + \alpha_r)} |h|_\infty +  \frac{1}{\alpha_r} |g|_\infty$$

For the surjectivity, unlike before, the literature does not give any existence theorem for such a linear problem, so we have to do it by ourselves : just observe that the estimate above gives that $\tilde{\mathscr T}$ has closed range. Indeed, if $\tilde{\mathscr T}(\phi_n,\psi_n) = (g_n,h_n)$ and $(g_n,h_n)$ converges in $\mathcal C^\alpha$ to a $(g,h)$, then by the above estimate and the Cauchy criteria, $(\phi_n,\psi_n)$ converges uniformly to a bounded continuous $(g,h)$. But also $\tilde{\mathscr T}(\phi_n,\psi_n)$ is bounded in $\mathcal C^\alpha$, so by regularity as in Proposition \ref{schaudersystem}, $(\phi_n,\psi_n)$ converges up to extraction and diagonal process in $\mathcal C^{2,\beta}$. Uniqueness of the limit implies that $(\phi,\psi)$ is indeed $\mathcal C^{2,\beta}$ and the convergence holds in the $\mathcal C^{2,\beta}$ sense. Finally, passing to the limit we get $\tilde{\mathscr T}(\phi,\psi) = (g,h)$ and $(\phi,\psi) \in \mathcal C^{2,\alpha}$ so that $(g,h)$ lies in $R(\tilde{\mathscr T})$. 
Finally, observe that the above estimate also holds for the formal adjoint of $\tilde{\mathscr T}$. Then, since the operators have smooth coefficients, the duality can be obtained thanks to the formal adjoint, and so $\Im \tilde{\mathscr T} = \overline{\Im \tilde{\mathscr T}} = \ker(\tilde{\mathscr T}^*)^\perp = \mathcal C^\alpha$.

The only thing left to see is that solving something for tilded operators really yields something back in the untilded world : what we mean is that since $w$ has exponential growth as $x\to +\infty$, we might have a problem. Indeed, we wanted to solve $\mathscr Tu = (g,h)$ in the weighted spaces, so we saw this equation as $\frac{1}{w}  \mathscr T(w\times(\phi,\psi)) = \frac{1}{w}(g,h) \in \mathcal C^\alpha(\mathbb R) \times \mathcal C^\alpha(\Omega_L)$ and obtained a solution $(\phi,\psi)\in \mathcal C^{2,\alpha}$. In the former cases, since $w\in\mathcal C^{\infty,\alpha}$ we did not have any problem to claim that also $wv\in\mathcal C^{2,\alpha}$ but here it is not the case any more, $w$ is not even bounded, and we might not have $w\psi \in \mathcal C^{2,\alpha}$, we might even not have that it is bounded. Actually, $\mathcal C^{2,\alpha}$ and boundedness for $w\times(\phi,\psi)$ are equivalent because of Schauder estimates, so we just have to see that it is indeed bounded. We will do that by showing that 
$\phi, \psi$ have actually a $Ce^{-rx}$ decay as $x\to +\infty$. 

For this, let $K = \max (|g|_\infty, |h|_\infty)$ and observe that if $A \geq \max( K, K/-f'(1))$ then 
$$(\overline\phi, \overline\psi) = \left(\dfrac{A}{\mu} e^{-rx}, Ae^{-rx}\right)$$ is a supersolution of \eqref{Ttilde} on $x > 1$, where \eqref{Ttilde} has constant coefficients and a positive $0$-order term. Now just multiply this supersolution by a  constant large enough so that it is above $(\phi, \psi)$ on $x=1$ and apply the usual maximum principle and compactness argument to $(\overline\phi - \phi, \overline\psi - \psi)$ : it can neither reach a negative minimum, nor have a negative infimum as $x \to +\infty$, which yields that $\phi, \psi \leq Ce^{-rx}$ for some constant $C > 0$. The same argument works for finding $C' < 0$ such that $\phi, \psi \geq -C' e^{-rx}$.

\section{The case $\varepsilon \simeq 0$}
\label{pertsing}
We start with $\left(c_w,\psi_w,\phi_w = \frac{1}{\mu}\psi_w(\cdot,0)\right)$. We want to continue this solution to a solution of \eqref{Seps} for small $\varepsilon > 0$. If we set as usual $\phi = \phi_0 + \varepsilon \phi_1, \psi = \psi_0 + \varepsilon\psi_1, c = c_0 + \varepsilon c_1$, using $$\phi_1 = \frac{\psi_1 + d\partial_y \psi_0 + \varepsilon d\partial_y\psi_1}{\mu}$$ from the exchange condition yields the equation

\begin{equation}
\label{singperturbbdcond}
\begin{split}
  & W\psi_1 + \varepsilon \frac{c_1}{\mu}\partial_x \psi_1 + \left(-\frac{\varepsilon D}{\mu}\partial_{xx} + \varepsilon\frac{c_0+c_1\varepsilon}{\mu}\partial_x\right)d\partial_y \psi_1 \\
  = &-\frac{c_1}{\mu} \partial_x\psi_0 - \left(-\frac{D}{\mu}\partial_{xx} + \frac{c_0+c_1\varepsilon}{\mu}\partial_x\right)d\partial_y \psi_0 
\end{split}
\end{equation}
as the upper boundary condition for the usual linearised problem in $\psi_1$ :

\begin{equation}
\label{singperturbeq}
-d\Delta\psi_1 + c_0\partial_x\psi_1 - f'(\psi_0)\psi_1 = -c_1\partial_x\psi_0 + R(\varepsilon,c_1,\psi_1)
\end{equation}
In particular, by taking $\varepsilon = 0$ in \eqref{singperturbbdcond}, \eqref{singperturbeq} we retrieve a linear Wentzell problem, i.e. \eqref{singperturbbdcond}, \eqref{singperturbeq} is a singular perturbation of a Wentzell problem on which we already applied the implicit function theorem. Conversely, we can see \eqref{singperturbbdcond} as an integro-differential regularisation of the Wentzell boundary condition, but the regularity theory of \cite{CS09} does not apply easily to this situation.

As before, we want to transform \eqref{singperturbbdcond} in a fixed Wentzell problem by using an auxiliary function. This time, since we do not have any existence or regularity theorem for such problems, we will have to compute everything by hand. Hopefully, since we work in a strip, we can use the partial (in $x$) Fourier transform which will be a very helpful tool. On the other hand, this time we will have to work with a constant coefficient operator instead of the linearised itself in order to be able to do the computations, but we will see that this is not a problem. From now on, let $w$ denote the same weight function as in the Wentzell section. We now give two simple technical lemmas that we will use throughout the next computations.

\begin{lem}
\label{lemfourier}
 If $k\in L^1$, $\hat k\in\mathcal C^\infty \cap L^2$ and $h\in L^\infty$, $\hat h\in\mathcal S'$ then the formula $$\mathcal F^{-1}(\hat k \hat h) = k*h$$ makes sense and holds (where $\mathcal F^{-1}$ denotes the inverse Fourier transform).
\end{lem}
\begin{proof}
Since $\hat k$ is a smooth function, the product distribution $\hat k \hat h$ makes sense and we can compute its inverse Fourier transform : we leave it to the reader to check the result using the classical properties of the Fourier transform on $L^2$ and the Fubini-Tonelli theorem.

%on $\phi\in\mathcal S$, $<\hat k \hat h, \check \phi> = <\hat h, \hat k \check \phi> = <h, \widehat{\hat k \check \phi}>$.
% But $\widehat{\hat k \check \phi} = k(-\cdot)*\phi$ since $\widehat{\hat k} = k(-\cdot)$ because $k\in L^2$ and the usual product-convolution formula holds since $\phi\in\mathcal S$. Thus, because $h\in L^\infty$, 
%
%$$<\hat k \hat h, \check \phi> = \int_{\mathbb R} h(u) \left(\int_{\mathbb R} k(t-u)\phi(t) dt \right) du = \int_{\mathbb R} \left( \int_{\mathbb R} k(t-u)h(u) du\right)\phi(t) dt = <k*h,\phi>$$ by Fubini-Tonnelli theorem.

\end{proof}

\begin{lem}
\label{lemconvexpdecay}
 Let $r > 0$. If $h\in L^\infty(\mathbb R)$ and $e^{-rx}h(x) \in L^\infty(\mathbb R)$ as well as $K\in L^1(\mathbb R)$ and $e^{-rt}K(t) \in L^1(\mathbb R)$ then $$K*h\in L^\infty(\mathbb R) \text{ and } e^{-rx}(K*h)(x) \in L^\infty(\mathbb R)$$
If moreover $e^{-rx}h\in\mathcal C^\alpha(\mathbb R)$, then $e^{-rx}(K*h)(x) \in \mathcal C^\alpha(\mathbb R)$.
\end{lem}
\begin{proof}
\begin{equation*}
\begin{split}
|e^{-rx}(K*h)(x)| \leq \int_{\mathbb R} |K(t)e^{-rx}h(x-t)|dt  &\leq \int_{\mathbb R} |e^{-rt}K(t)||e^{-r(x-t)}h(x-t)|dt \\ &\leq |e^{-rt}K(t)|_{L^1} |e^{-rx}h(x)|_{L^\infty}
\end{split}
\end{equation*}
For the second part of just observe that 
\begin{equation*}
\begin{split}
\frac{|e^{-ry}(K*h)(y)-e^{-rx}(K*h)(x)|}{|x-y|^\alpha} &\leq \int_{\mathbb R} K(t)e^{-rt} \frac{|e^{-r(y-t)}h(y-t)-e^{-r(x-t)}h(x-t)|}{|y-x|^\alpha}dt \\
& \leq |e^{-rt}K(t)|_{L^1} |e^{-rx}h(x)|_\alpha
\end{split}
\end{equation*}
\end{proof}

Now we assert the following :

\begin{lem}
 By taking $r>0$ small enough in the definition of $w$, we have $$\tilde\psi \in \mathcal C^1 (\mathbb [0,1]\times \mathbb R \times \mathcal C^{3,\alpha}_w(\Omega_L) ; \mathcal C^{2,\alpha}_w(\Omega_L))$$ where $u = \tilde\psi(\varepsilon,c_1,v)$ solves

\begin{center}
\begin{tikzpicture}
\draw (-8,0) -- (8,0) node[pos=0.5,above] {$Wu + \varepsilon \frac{c_1}{\mu}\partial_x u + \left(-\frac{\varepsilon D}{\mu}\partial_{xx} + \varepsilon\frac{c_0+c_1\varepsilon}{\mu}\partial_x\right)d\partial_y u = h_0 - \varepsilon \frac{c_1}{\mu}\partial_x v - \left(-\frac{\varepsilon D}{\mu}\partial_{xx} + \varepsilon\frac{c_0+c_1\varepsilon}{\mu}\partial_x\right)d\partial_y v$};
\node at (0,-1) {$\mathcal -\Delta u + u = 0$};
\draw (-8,-2) -- (8,-2) node[pos=0.5,below] {$\partial_y u = 0$};
\end{tikzpicture}
\end{center}
and $h_0 := -\frac{c_1}{\mu} \partial_x\psi_0 - \left(-\frac{D}{\mu}\partial_{xx} + \frac{c_0+c_1\varepsilon}{\mu}\partial_x\right)d\partial_y \psi_0 \in \mathcal C^{\alpha}(\mathbb R)$.  Moreover, we have the estimate

$$|u|_{\mathcal C^{2,\alpha}(\Omega_L)} \leq C_1|h_0|_\infty + C_2 \left|\frac{1}{\varepsilon}K_0\left(\frac{|x|}{d\varepsilon}\right)*(h_0 + \varepsilon h(v))\right|_\alpha + C_3 |h_0 + \varepsilon h(v)|_\alpha$$
where $K_0$ denotes the $0$-th modified Bessel function of the second kind (which is integrable\footnote{it increases in a logarithmic fashion as $x\to 0$ and decreases as $e^{-x}/x$ as $x\to\infty$, see \cite{OMS09} p.532.} and whose Fourier transform is $\frac{\pi}{\sqrt{1+x^2}}$) so that $\frac{1}{\varepsilon}K_0\left(\frac{|x|}{d\varepsilon}\right)$ realises an approximation to the identity, and where $h(v)$ denotes $\partial_x v + \partial_{xy} v + \partial_{xxy} v$. Finally, we also have $$\mathcal L\tilde\psi(\varepsilon,c_1,v) \in \mathcal C^{1,\alpha}_w(\Omega_L)$$
\end{lem}

\begin{proof}
 The proof is based on the kernel analysis of this problem after applying a partial Fourier transform. First, let us see that $v\in\mathcal C^{3,\alpha}_w(\mathbb R)$ implies that the right-hand side in the boundary condition for $u$ is in $\mathcal C^{\alpha}_w(\mathbb R)$. In the following, for the sake of notations we will only write it $h$.
Applying formally the $x$-Fourier transform, we get a one parameter (in $\xi$) family of two-points boundary problems (in $y$) which are solved necessarily by $$\hat u(\xi,y) = C(\xi)\cosh(\sqrt{\xi^2 + 1}(y+L))$$ and the upper boundary condition yields, if we set $\beta(\xi) = \sqrt{\xi^2 + 1}$

$$ C(\xi) = \frac{\hat h(\xi)}{d\beta\left(\xi\right)\sinh\left(\beta\left(\xi\right)L\right)\left(1+\frac{\varepsilon D}{\mu}\xi^2 + \varepsilon \frac{c_0 + c_1\varepsilon}{\mu}i\xi\right) + \left(\frac{D}{\mu}\xi^2 + \frac{c_0+c_1\varepsilon}{\mu}i\xi\right)\cosh\left(\beta\left(\xi\right)L\right)} $$
i.e. we get
$$\hat u(\xi,y) = C(\xi)\cosh\left(\beta\left(\xi\right)\left(y+L\right)\right) \hat h(\xi) =: \hat k_y(\xi)\hat h(\xi)$$
where $$ C(\xi) = \frac{1}{d\beta\left(\xi\right)\sinh\left(\beta\left(\xi\right)L\right)\left(1+\frac{\varepsilon D}{\mu}\xi^2 + \varepsilon \frac{c_0 + c_1\varepsilon}{\mu}i\xi\right) + \left(\frac{D}{\mu}\xi^2 + \frac{c_0+c_1\varepsilon}{\mu}i\xi\right)\cosh\left(\beta\left(\xi\right)L\right)} $$

Now for each $ -L \leq y < 0$, this kernel is in the Schwartz space $\mathscr S(\mathbb R)$ and $u(x,y)$ for such $y$ can be obtained by the usual convolution product between the Fourier inverse of $\hat k_y(\xi)$ and $h(x)$. Moreover, since for $-L \leq y < -\delta$ with $\delta > 0$ the kernels are a $\mathcal C^\infty$ family that is uniformly bounded in the Schwartz space $\mathscr S(\xi)$, we have by dominated convergence that $u$ is a $\mathcal C^\infty$ function in $\Omega_L$, in particular it is locally $\mathcal C^{2,\alpha}$. We now want to investigate the regularity of $u$ on the line $y=0$ in order to  use Schauder estimates to conclude to a uniform $\mathcal C^{2,\alpha}$ regularity.

On $y=0$, things get a little more complicated since the kernel involved is $$ \hat k_0(\xi) = \frac{1}{d\beta\left(\xi\right)\tanh\left(\beta\left(\xi\right)L\right)\left(1+\frac{\varepsilon D}{\mu}\xi^2 + \varepsilon \frac{c_0 + c_1\varepsilon}{\mu}i\xi\right) + \left(\frac{D}{\mu}\xi^2 + \frac{c_0+c_1\varepsilon}{\mu}i\xi\right)} $$ which decays only like $\dfrac{1}{1 + \frac{D}{\mu}\xi^2 + \frac{\varepsilon d D}{\mu}|\xi|^3}$, and $(i\xi)^2 \hat k_0(\xi)$ like $\dfrac{\xi^2}{1 + \frac{D}{\mu}\xi^2 + \frac{\varepsilon d D}{\mu}|\xi|^3}$. Keep in mind that we are interested in $\varepsilon$ independent estimates, so we cannot use the little bonus decay it gives. Nonetheless, observe that $\hat k_0$ is $\mathcal C^1$ with respect to the parameters $(c_1\in\mathbb R,\varepsilon\in[0,1])$ (this is something we will need in the end to apply the implicit function theorem) and decays at worst (when $\varepsilon = 0$) as $\frac{\mu}{D(1+\xi^2)}$. Heuristically, we see that $\varepsilon > 0$ is not a problem in 
the sense that it adds decay and does not prevent analyticity, so in a Fourier point of view, the worst case is when $\varepsilon = 0$, and in this case the kernels are nothing more than the kernels for the Wentzell problem in a strip, which is known to be well-posed. We use a Paley-Wiener type theorem to prove this :

\begin{itemize}
 \item $\hat k_0(\xi)$ is a $\mathcal C^1$ in $(c_1 \in \mathbb R ,\varepsilon\in [0,1])$ family of integrable (because the worst decay is $\frac{1}{1+\frac{D}{\mu}\xi^2}$ for $\varepsilon = 0$) and real analytic functions (as the inverse of real analytic functions that have no zero). Moreover, independently from $\varepsilon$ and $c_1$, these real analytic functions admit an analytic continuation to a complex strip $|\Im \zeta| < a$ with $a > 0$ that have a $\eta$-uniformly bounded $L^1$ norm on the real lines $\mathbb R + i\eta$, $-a < \eta < a$, see lemma \ref{analcont}. By virtue of the Paley-Wiener type theorem of \cite{RS} (IX.14) and the dominated convergence theorem we know that $k_0(x)$ is a $\mathcal C^1$ in $c_1 \in \mathbb R ,\varepsilon\in [0,1]$ family of bounded continuous real functions that satisfy all $|k_0(x)| \leq C_a e^{-a|x|}$. Now, we can say that $u(x,0) = k_0*h$ is a bounded 
continuous function that is $\mathcal C^1$ with respect to the parameters $\varepsilon, c^1$ and $v$ (since $h$ is $\mathcal C^1$ in those parameters as product and sum of affine functions).

\item For the sake of simplicity, we divide the analysis of $\xi^2 \hat k_0(\xi)$ in two cases : $\varepsilon > 0$ or $\varepsilon = 0$ and we will see that the result is smooth in $\varepsilon$. \\
Case $\varepsilon = 0$ : in this case, the asymptotic behaviour of  $\xi^2 \hat k_0(\xi)$ as $|\xi| \to \infty$ yields $\xi^2 \hat k_0(\xi) = \frac{\mu}{D} - \frac{\mu^2d}{D^2\sqrt{1+\xi^2} } + r_1(\xi)$ where $r_1$ denotes an integrable function (it decays like $1/\xi^2$) that has an analytic continuation in some complex strip $|\Im| < a$, i.e. to which the same analysis as above applies. Thus, the Fourier transform of $\xi^2 \hat k_0(\xi)$ is given by $\frac{\mu}{D}\delta - \frac{\mu^2d}{D^2}\frac{1}{\pi}K_0(|x|) + \check r_1$, where $\delta$ denotes the Dirac distribution, $K_0$ the modified Bessel function of order $0$, and where $\check r_1$ has the properties described in the section above. By lemma \ref{lemfourier} we get
$$\partial_{xx}u(x,0) = \frac{\mu}{D}h_0 - \frac{\mu^2d}{D^2}\frac{1}{\pi}K_0(|\cdot|)*h_0 + \check r_1*h_0 \in \mathcal C^\alpha(\mathbb R)$$
since $h_0 \in \mathcal C^\alpha(\mathbb R)$  \\

Case $\varepsilon > 0$. This changes the decay of the kernel from constant to $1/\xi$, so we will not get a Dirac term in the Fourier transform. Nonetheless, what is tricky is that we want $u=\tilde\psi$ to be a $\mathcal C^1$ function in $\varepsilon$ to be able to use the implicit function theorem, i.e. we separated the computations for $\varepsilon > 0$ or $=0$, but in the end the results should agree when $\varepsilon \to 0$. This will be based on the fact that the functions we will obtain will behave as an approximation to the identity as $\varepsilon\to 0$.

Indeed,
$\xi^2 \hat k_0(\xi) = \frac{\mu}{D} \frac{1}{\sqrt{1+(d\varepsilon)^2\xi^2}} + r_2$. Notice that we chose to put $\varepsilon$ in front of $\xi$ inside the square root rather than just let it appear as $\frac{1}{\varepsilon}$ : this is the right way to get smoothness in $\varepsilon$, since this gives the correct decay even if $\varepsilon = 0$. Now, observe that the inverse Fourier transform of the first term is $\frac{\mu}{D}\frac{1}{\pi}\frac{1}{d\varepsilon}K_0(\frac{|x|}{d\varepsilon})$ : since $K_0(|\cdot|)$ is an integrable function on $\mathbb R^1$ whose integral equals to $\pi$, this clearly is $\frac{\mu}{D}$ times an approximation to the identity. 
%attention conneries, tenter de détailler mieux...
We finish by saying that the term $r_2$ can be computed as $\frac{C(\varepsilon)}{\sqrt{1+\xi^2}} + r_3$ where $C(\varepsilon)$ is a smooth function that satisfies $C(0) = -\frac{\mu^2d}{D}$ and $r_3$ is a smooth family with respect to $(c_1,\varepsilon)$ of integrable functions to which the same analysis as $r_1$ applies, and that goes to $r_1$ as $\varepsilon \to 0$.
\end{itemize}
This analysis gives that $u\in\mathcal C^{2,\alpha}(y=0)$ and then by applying Schauder estimates for the Dirichlet problem, we get that $u\in\mathcal C^{2,\alpha}(\Omega_L)$. We describe now with more details the same technique applied on $w_1u$.

We are now left to show that a weighted data yields a weighted solution, i.e. that $w_1u\in\mathcal C^{2,\alpha}(\Omega_L)$. We observe that $v = w_1u$ solves the following equation in $\Omega_L$ :

\begin{center}
\begin{tikzpicture}
\draw (-6,0) -- (6,0) node[pos=0.5,above] {$v = v$};

\node at (0,-1) {\ \ \ \ \ \ \ \ \ $-\Delta v  + 2\left(\frac{\partial_x w_1}{w_1}\right)\partial_x v + \left(1+\frac{\partial_{xx} w_1}{w_1} - 2\left(\frac{\partial_x w_1}{w_1}\right)^2\right)v = 0$};

\draw (-6,-2) -- (6,-2) node[pos=0.5,below] {$\partial_y v = 0$};

\draw[<->] (-6.3,-2) -- (-6.3,0) node[pos=0.5,left] {$L$};

\filldraw [fill=gray!20,draw=black]
(-4.5,-0.2) node[below] {$B$}
(-3.5,0) arc (180:0:-1)
(-5.5,0) -- (-3.5,0);

\filldraw [fill=gray!20,draw=black]
(-5.5,-2) arc (180:0:1)
(-5.5,-2) -- (-3.5,-2);

\draw[blue](-5.5,-2) -- (-3.5,-2);
\draw[blue](-5.5,0) -- (-3.5,0) node [pos=0.5,above] {$T$};
\end{tikzpicture}
\end{center}
Thanks to the expression of $w_1$, the coefficients of this equation are smooth bounded functions and we can use local estimates up to the boundary for the Dirichlet or the Neumann problem (see Cor. 6.7 and Lemma 6.29 in \cite{GT}), so it suffices to show that $w_1u$ is bounded and that $w_1u(\cdot,0) \in \mathcal C^{2,\alpha}(\mathbb R)$, which thanks to the expression of $w_1$, is similar to $w_1 \partial_{xx} u \in \mathcal C^{\alpha}(\mathbb R)$. We show that these are true provided $r < \min(\rho,1)$ (see lemma \ref{analcont} for the definition of $\rho$).
\begin{itemize}
 \item $w_1u$ is bounded thanks to lemma \ref{lemconvexpdecay} : indeed $w_1u(x,y) = w_1(x)(k_y*h)(x)$. As we already said, $k_y$ is a family of bounded continuous functions uniformly bounded in $L^1$. Moreover, they have a uniform $Ce^{-\rho |x|}$ decay as $x\to\pm\infty$ : for this see lemma \ref{analcont} below and use \cite{RS}, Theorem IX.14.
 \item $w_1\partial_{xx} u(\cdot,0)$ is bounded and has $\mathcal C^\alpha$ regularity since $K_0(|\cdot|)$ has $e^{-|x|}/|x|$ decay and the other kernels appearing in $\partial_{xx} u(\cdot,0)$ satisfy lemma \ref{lemconvexpdecay} too, thanks to their common analyticity ; see lemma \ref{analcont}.
\end{itemize}
\end{proof}

\begin{lem}
\label{analcont}
 Replacing $\xi$ with the complex variable $\zeta$ in $\hat k_y(\xi)$ yields a meromorphic continuation of $\hat k_y$ in the strip $-1 < \Im z < 1$ that has no pole in a strip $-\rho < \Im z < \rho$ for $\rho > 0$ small enough.
\end{lem}

\begin{proof}
 First, observe that apart from $\beta(\xi)$, the denominator of $\hat k_y$, which we will note $F(\xi)$ in this proof : $$d\beta\left(\xi\right)\sinh\left(\beta\left(\xi\right)L\right)\left(1+\frac{\varepsilon D}{\mu}\xi^2 + \varepsilon \frac{c_0 + c_1\varepsilon}{\mu}i\xi\right) + \left(\frac{D}{\mu}\xi^2 + \frac{c_0+c_1\varepsilon}{\mu}i\xi\right)\cosh\left(\beta\left(\xi\right)L\right)$$ is composed of holomorphic functions over the whole complex plane. The only limiting function is $\beta(\xi)$ which is holomorphic in the strip $-1 < \Im z < 1$. As a result, $\hat k_y$ is meromorphic in this strip. Moreover, thanks to the $\frac{D}{\mu}\xi^2\cosh(\beta(\xi)L)$ term we can see that if $\xi$ is large enough, $|F(\xi)|$ is large enough (independently from $\Im \zeta$ in the strip and from $c_1, \varepsilon$), so its zeroes have to be in a rectangle centred at the complex origin whose length depends on the parameters (but not on $\varepsilon$ or $c_1$). Since the zeros of a non-zero holomorphic functions 
are isolated, we know that $F$ has a finite number of zeros in such a rectangle. Moreover, a direct computation shows that it cannot have any zero on the real line. Thus, there exists $\rho >0$ small enough such that on the strip $-\rho < \Im \xi < \rho$, $F$ does not vanish.\footnote{Actually, contour integrals of $F'/F$ show that $F$ has only one or two zeros, depending on $d,D,\mu,c_0$ and $L$, and we can easily see that these are on the imaginary axis by solving for $\zeta = i\eta$ the equation $d\sqrt{1-\eta^2}\tanh(\sqrt{1-\eta^2} L)(1-\varepsilon\frac{D}{\mu}\eta^2-\varepsilon\frac{c_0+c_1\varepsilon}{\mu}\eta) =\frac{D}{\mu}\eta^2+\frac{c_0+c_1\varepsilon}{\mu}\eta $.}

\end{proof}

We now turn to the implicit function theorem procedure that concludes this section. Searching as usual for $\psi^1$ with form $\psi^1 = v + \tilde\psi(v)$, we are reduced to solving the following problem on $v$ :

\begin{center}
\begin{tikzpicture}
\draw (-6,0) -- (6,0) node[pos=0.5,above] {$d\partial_yv - \frac{D}{\mu}\partial_{xx} v +
\frac{c_0}{\mu}\partial_x v = 0$};

\node at (0,-1) {$\mathcal Lv + c^1\partial_x\psi^0= R(\varepsilon,c^1,v) - \mathcal L\tilde\psi(s,c^1,v)$};

\draw (-6,-2) -- (6,-2) node[pos=0.5,below] {$\partial_y v = 0$};

\end{tikzpicture}
\end{center}

We now use the analysis of section \ref{pwopen} to claim that $\mathcal L$ endowed with this Wentzell boundary condition has the Fredholm property of index $0$ between $\mathcal C^{3,\alpha}_w$ and $\mathcal C^{1,\alpha}_w$.  Since $R$ lies also in $\mathcal C^{1,\alpha}_w$, the procedure is then exactly the same as in \ref{pwopen} and for $\varepsilon > 0$ small enough leads to a solution $c_1, \psi_1 = v+\tilde{\psi}(\varepsilon,c_1,v)$ of \eqref{singperturbbdcond}, \eqref{singperturbeq} that lies in $\mathcal C^{2,\alpha}_w(\Omega_L)$.

Then, setting $\phi_1(x) = \dfrac{\psi_1(x,0) + d\partial_y \psi_0(x,0) + \varepsilon d\partial_y\psi_1(x,0)}{\mu}$ we get that $\phi_1 \in \mathcal C^{2,\alpha}_w(\mathbb R)$ and that $\phi = \frac{1}{\mu}\psi_0(x,0)+\varepsilon\phi_1, \psi = \psi_0 + \varepsilon\psi_1$ solves \eqref{Seps} except for the right limit condition.
But the analysis of Proposition \ref{rightlimitcd} with the maximum principle and Hopf lemma for the system gives that if $\varepsilon > 0$ is taken small enough, this limit holds.

\begin{rmq}
Observe that solving this singular perturbation had a price of one derivative: we started by assuming $\psi_0 \in \mathcal C^{3,\alpha}(\Omega_L)$ but we end up with a solution of $(S)_\varepsilon$ that is only $\mathcal C^{2,\alpha}(\mathbb R)\times\mathcal C^{2,\alpha}(\Omega_L)$.
\end{rmq}

\section*{Acknowledgement}
The research leading to these results has received funding from the European Research Council under the European Union’s Seventh Framework Programme (FP/2007-2013) / ERC Grant Agreement n.321186 - ReaDi -Reaction-Diffusion Equations, Propagation and Modelling. I also would like to thank Professors H. Berestycki and J.M. Roquejoffre for their support and their help during the preparation of this work.

\bibliographystyle{amsplain}
\bibliography{refs}

\end{document}